\documentclass[11pt]{amsart}
\usepackage{amssymb}
\usepackage{mathrsfs}
\usepackage[all,cmtip]{xy}
\usepackage{color}

\numberwithin{equation}{section}

\newtheorem{thm}{Theorem}[section]
\newtheorem{prop}[thm]{Proposition}
\newtheorem{lemma}[thm]{Lemma}
\newtheorem{cor}[thm]{Corollary}

\theoremstyle{definition}

\newtheorem*{rem}{Remark}

\newcommand{\ds}{\displaystyle}

\newcommand{\surj}{\twoheadrightarrow}
\newcommand{\embed}{\hookrightarrow}
\newcommand{\Bl}{{\rm Bl}}
\newcommand{\Sym}{{\rm Sym}}
\newcommand{\tor}{{\rm tor}}
\newcommand{\e}{\varepsilon}
\newcommand{\id}{{\rm id}}

\newcommand{\Gr}{{\rm Gr}}
\newcommand{\BBS}{\mathbb{S}}

\newcommand{\CD}{\xymatrix@R=1pc@C=1pc}
\newcommand{\CDR}{\xymatrix@R=1pc}
\newcommand{\CDC}{\xymatrix@C=1pc}

\DeclareMathOperator{\QQuot}{{\it\mathfrak{Q}uot}} \DeclareMathOperator{\Quot}{Quot}
 
\DeclareMathOperator{\EExt}{{\it\mathcal{E}\!xt}} \DeclareMathOperator{\Ext}{Ext}
\DeclareMathOperator{\TTor}{{\it\mathcal{T}\!or}} 
\DeclareMathOperator{\HHom}{{\it\mathcal{H}\!om}} \DeclareMathOperator{\Hom}{Hom}
\DeclareMathOperator{\rank}{rank} \DeclareMathOperator{\Coker}{Coker}
\DeclareMathOperator{\Supp}{Supp} \DeclareMathOperator{\Spec}{Spec}
\DeclareMathOperator{\PProj}{{\bf Proj}} 
\DeclareMathOperator{\Img}{Im} \DeclareMathOperator{\Mor}{Mor}
\DeclareMathOperator{\codim}{codim} 

\title[A Compactification of the Space of Parametrized Rational Curves]
{A Compactification of the Space of Parametrized Rational Curves in Grassmannians}
\date{}

\author{Yijun Shao}
\address{Department of Mathematics and Computer Science, Wesleyan University, CT 06459}
\email{yshao@wesleyan.edu}

\begin{document}

\maketitle

\begin{abstract}
We construct an explicit compactification for the space of parametrized rational curves in a Grassmanian by a nonsingular projective variety such that the boundary is a divisor with simple normal crossings. This compactification is obtained by blowing up a Quot scheme successively along appropriate subschemes.
\end{abstract}

\section{Introduction}

Let $V$ be a vector space of dimension $n$ over an algebraically closed field $\Bbbk$ and denote by $\Gr(k,V)$ the Grassmannian of $k$-dimensional subspaces of $V$. By a parametrized rational curve in $\Gr(k,V)$ we mean a morphism from $\mathbb{P}^1$ to $\Gr(k,V)$.  The space $\Mor_d(\mathbb{P}^1,\Gr(k,V))$ of all morphisms of degree $d$ from $\mathbb{P}^1$ to $\Gr(k,V)$ is a smooth quasi-projective variety for each $d\geq 0$, and is non-compact for $d\geq 1$.
In this paper, we provide an explicit compactification for this space which satisfies the following conditions: the compactification is a nonsingular projective variety and the boundary is a divisor with simple normal crossings. This extends the result in \cite{HLS}, where maps to a projective space are considered.

The main tools of the construction include (i) a Quot scheme which serves as a smooth compactification of $\Mor_d(\mathbb{P}^1,\Gr(k,V))$, and (ii) a sequence of {blowups} which turns the boundary into a divisor with simple normal crossings. The method of construction is as follows. A smooth compactification of $\Mor_d(\mathbb{P}^1,\Gr(k,V))$ is given by the Quot scheme $\Quot^{n-k,d}_{V_{\mathbb{P}^1}/\mathbb{P}^1/\Bbbk}$, which parametrizes all quotient sheaves of the trivial vector bundle $V_{\mathbb{P}^1}$ on $\mathbb{P}^1$ of rank $n-k$ and degree $d$. For short notations, we introduce $Q_d:=\Quot^{n-k,d}_{V_{\mathbb{P}^1}/\mathbb{P}^1/\Bbbk}$ and $\mathring{Q}_d:=\Mor_d(\mathbb{P}^1,\Gr(k,V))$. We define a chain of closed subschemes in the boundary $Q_d\setminus\mathring{Q}_d$:
\[
Z_{d,0}\subset Z_{d,1}\subset \cdots \subset Z_{d,d-1}=Q_d\setminus \mathring{Q}_d.
\]
Set-theoretically, $Z_{d,r}$ is the locus of the quotients whose torsion part has degree (or length) at least $d-r$; its scheme structure is given as a determinantal subscheme. We then blow up the Quot scheme successively along these closed subschemes $Z_{d,r}$. Starting with $Q_d^{-1}:=Q_d$, $Z_{d,r}^{-1}:=Z_{d,r}$, we let $Q_d^r$ ($r=0,\dots,d-1$) be the blowup of $Q_d^{r-1}$ along $Z_{d,r}^{r-1}$, $Z_{d,r}^r$ the exceptional divisor, and $Z_{d,l}^r$ the proper transform of $Z_{d,l}^{r-1}$ in $Q_d^r$ for $l\neq r$. Set $\widetilde Q_d:=Q_d^{d-1}$ and $\widetilde Z_{d,r}:=Z_{d,r}^{d-1}$. Our goal is to prove
\begin{thm}\label{Goal}
The final outcome $\widetilde Q_d$ is a compactification of $\mathring{Q}_d=\Mor_d(\mathbb{P}^1,\Gr(k,V))$ satisfying the following properties:
\begin{enumerate}
\item $\widetilde Q_d$ is an irreducible nonsingular projective variety;
\item $\widetilde Z_{d,r}$'s are irreducible nonsingular subvarieties of codimension one and they intersect transversally, so that the boundary $\widetilde Q_d\setminus \mathring{Q}_d=\bigcup_{r=0}^{d-1}\widetilde Z_{d,r}$ is a divisor with simple normal crossings.
\end{enumerate}
\end{thm}

This way of producing such type of compactifications was used in many classical examples such as the construction of the space of complete quadrics \cite{Vai82}, the space of complete collineations \cite{Vain84}, Fulton-MacPherson compactification \cite{FM94}, MacPherson-Procesi compactification \cite{MP98}, etc. Our construction is much closer to that of the spaces of complete objects (quadrics, etc.), and in fact, we used in our proof a result from Vainsencher's construction of complete collineations. Because of the strong similarity between our compactification and those spaces of complete objects, we believe that our compactification is also a parameter space of similar kind, although the interpretation of the objects that this space parametrizes can be quite subtle. This is a direction being pursued for a future publication \cite{HS}.

Although the method of the construction is straightforward to describe, the proof is more involved. We outline some key steps of the proof here. In $\S 3$, we endow the sets $Z_{d,r}$ with the structure of determinantal subschemes. This is done with the help of the universal exact sequence of $Q_d$: $0\to \mathcal{E}_d\to V_{\mathbb{P}^1_{Q_d}}\to \mathcal{F}_d\to 0$, where $\mathcal{E}_d$ is locally free. For each integer $m$, there is an induced homomorphism $\rho_{d,m}:  \pi_*V_{\mathbb{P}^1_{Q_d}}^\vee(m)\to \pi_*\mathcal{E}_d^\vee(m)$. Then $Z_{d,r}$ is defined to be the zero locus of the exterior power $\bigwedge^{k(m+1)+r+1}\rho_{d,m}$ for any $m\gg 0$. Let $\mathring{Z}_{d,r}:=Z_{d,r}\setminus Z_{d,r-1}$ be the open subscheme of $Z_{d,r}$. We can show that the locally closed subschemes $Z_{d,0}$, $\mathring{Z}_{d,1}$, $\cdots$, $\mathring{Z}_{d,d-1}$, $\mathring{Q}_d$ form a flattening stratificaiton of $Q_d$ by the sheaf $\mathcal{F}_d^\e:=\EExt^1(\mathcal{F}_d,\mathcal{O}_{\mathbb{P}^1_{Q_d}})$. In $\S 4$, we show that there is a natural morphism from an irreducible nonsingular variety $Q_{d,r}$ to $Q_{d}$ whose (set-theoretic) image is $Z_{d,r}$ for each $r$. Here $Q_{d,r}$ is a relative Quot scheme over $Q_r$ parametrizing torsion quotients of $\mathcal{E}_r$ that are flat over $Q_r$ with degree $d-r$. This morphism maps the open subset $Q_{d,r}\times_{Q_r}\mathring{Q}_r$ of $Q_{d,r}$ isomorphically onto $\mathring{Z}_{d,r}$, based on the fact that $\mathcal{F}_d^\e$ is flat over $\mathring{Z}_{d,r}$. In $\S 5$ we first describe the procedure of {blowups}, and then prove some main theorems. We first show that the blowup along $Z_{d,r}^{r-1}$ (the proper transform of $Z_{d,r}$ in $Q_d^{r-1}$) is the same as the blowup along the total transform of $Z_{d,r}$ in $Q_d^{r-1}$. This allows us to embed the blowup $Q_d^r$ into a product of some spaces of complete collineations and then set up a key commutative diagram (see Diagram (\ref{ZembedDiagram})). Using this diagram, we can show that the proper transform $Z_{d,r}^{r-1}$ is isomorphic to $Q_{d,r}\times_{Q_r}Q_r^{r-1}$. It is now clear that we should use induction to prove that $Z_{d,r}^{r-1}$ and $Q_d^r$ are nonsingular. The proof for the transversality of the intersections of $Z_{d,r}^{d-1}$'s in $Q_d^{d-1}$ is also based on the isomorphism between $Z_{d,r}^{r-1}$ and $Q_{d,r}\times_{Q_r}Q_r^{r-1}$.

\textbf{Acknowledgements.} The author is very grateful to his Ph.D advisor Yi Hu for guidance and support. He is also thankful to Ana-Maria Castravet, Kirti Joshi, Douglas Ulmer and Dragos Oprea for encouragement and helpful discussions.

\section{Preliminaries}

In this section, we will first give a brief review of Quot schemes, and then we fix some notations and review some basic facts.

Let $H$ be a sheaf over a scheme $X$. A quotient of $H$ is a sheaf $F$ on $X$ together with a surjective homomorphism $H\surj F$. Two quotients $H\surj F_1$ and $H\surj F_2$ are said to be \emph{equivalent} if there is an isomorphism $F_1\simeq F_2$ such that the following diagram commutes.
\[
\CDR{ %
 H\ar@{->>}[r]\ar@{=}[d] & F_1 \ar[d]^{\simeq}\\
 H\ar@{->>}[r] & F_2 \\
} %
\]
Equivalently, the two quotients are equivalent if and only if their kernels are equal as subsheaves of $H$. We denote by $[H\surj F]$ the equivalence class of a quotient $H\surj F$.

Let $S$ be a noetherian scheme, $X$ a projective $S$-scheme, and $H$ a coherent sheaf on $X$. For any $S$-scheme $T$, we denote by $H_T$ the pullback of $H$ by the projection $X\times_S T\to X$. Given a very ample line bundle $L$ on $X$ relative to $S$ and a polynomial $P(t)\in\mathbb{Q}[t]$, we define a contravariant functor, denoted by $\QQuot^{P,L}_{H/X/S}$ and called the Quot functor, from the category of $S$-schemes to the category of sets as follows: for any $S$-scheme $T$, let $\QQuot^{P,L}_{H/X/S}(T)$ be the set of equivalence classes of quotients $H_T\surj F$ where $F$ is flat over $T$ with Hilbert polynomial $P$ (relative to $L$).

This functor is represented by a projective $S$-scheme \cite{Gro}, called the (relative) Quot scheme, which we denote by $\Quot^{P,L}_{H/X/S}$. It is equipped with a \emph{universal quotient}
\[
\xymatrix{
  H_{\Quot^{P,L}_{H/X/S}}\ar@{->>}[r] & \mathcal{F},
}
\]
on $X\times_S \Quot^{P,L}_{H/X/S}$, where $\mathcal{F}$ is flat over $\Quot^{P,L}_{H/X/S}$ with $P$ as its Hilbert polynomial. Sometimes it will bring us convenience to add to the universal quotient its kernel to form the \emph{universal exact sequence}:
\[
0\to \mathcal{E}\to H_{\Quot^{P,L}_{H/X/S}}\to \mathcal{F}\to 0.
\]
\begin{thm}[Universal Property of Quot schemes]
For any $S$-scheme $Y$, any quotient $H_Y\surj \mathcal{Q}$ on $X\times_SY$ with $\mathcal{Q}$ flat over $Y$ with Hilbert polynomial $P$ determines a unique $S$-morphism $f: Y\to \Quot^{P,L}_{H/X/S}$ such that the pullback of the universal quotient by the induced morphism $\bar f: X\times_SY \to X\times_S\Quot^{P,L}_{H/X/S}$ is equivalent to the given quotient on $X\times_SY$:
\[
\CD{
  H_Y\ar@{->>}[rr]\ar@{=}[d] & & \mathcal{Q} \ar[d]^{\simeq}\\
  H_Y\ar@{=}[r] & \bar f^*H_{\Quot^{P,L}_{H/X/S}}\ar@{->>}[r] & \bar f^*\mathcal{F} \\
}
\]
\end{thm}
A consequence of the uniqueness of the morphism is: if $f_1, f_2: Y\to \Quot^{P,L}_{H/X/S}$ are two $S$-morphisms such that the pullbacks of the universal quotient are equivalent:
\[
\xymatrix@R=8pt@C=1pc{
  H_Y\ar@{=}[r]\ar@{=}[d] & \bar f_1^*H_{\Quot^{P,L}_{H/X/S}}\ar@{->>}[r] & \bar f_1^*\mathcal{F} \ar[d]^{\simeq}\\
  H_Y\ar@{=}[r] & \bar f_2^*H_{\Quot^{P,L}_{H/X/S}}\ar@{->>}[r] & \bar f_2^*\mathcal{F} \\
}
\]
then $f_1=f_2$. Note that it is not sufficient to claim $f_1=f_2$ if we only know an isomorphism $\bar f_1^*\mathcal{F}\simeq \bar f_2^*\mathcal{F}$. The commutativity of the above diagram is crucial.

In the special case that $X=S$ and $P(t)$ is a constant integer $r$, the Quot scheme becomes a (relative) Grassmannian, which we denote by $\Gr_S(H,r)$. If in addition $H$ is locally free, we also denote this Grassmannian by $\Gr_S(m,H)$ where $m=\rank H-r$.

Most references (for example, \cite{Huy,Nitsure,Ser06}) prove the representability of the Quot functor by constructing the Quot scheme as a closed subscheme of a Grassmannian. Therefore a natural consequence of this construction is the embedding of the Quot scheme into a Grassmannian, but this fact is rarely isolated as a theorem. For convenience of reference, we state this fact below.
\begin{thm}[Embedding of Quot into Grassmannian]\label{quotembedding}
Suppose $H$ is flat over $S$ with Hilbert polynomial $R(t)\in \mathbb{Q}[t]$ and put $Q:=\Quot^{P,L}_{H/X/S}$. Let $\pi: X\to S$ and $\sigma: Q\to S$ be the structure morphisms, and $\pi_Q: X\times_SQ\to Q$ the projection. There is an integer $N$ such that for all $m\geq N$
\begin{enumerate}
\item $\pi_{Q*}(H_Q(m))=\sigma^*\pi_{*}(H(m))$,

\item $\pi_{Q*}(H_Q(m))$ and $\pi_{Q*}(\mathcal{F}(m))$ are locally free of rank $R(m)$ and $P(m)$ respectively, and the induced homomorphism $\pi_{Q*}(H_Q(m))\to \pi_{Q*}(\mathcal{F}(m))$ is surjective, and

\item the morphism $Q\to\Gr_S(\pi_{*}(H(m)),P(m))=\Gr_S(R(m)-P(m),\pi_{*}(H(m)))$ induced by the surjective homomorphism $\sigma^*\pi_{*}(H(m))\to \pi_{Q*}(\mathcal{F}(m))$ is a closed embedding.
\end{enumerate}
\end{thm}

The Quot scheme has a base-change property: for any $S$-scheme $T$, $\Quot^{P,L}_{H/X/S}\times_S T=\Quot^{P,L_T}_{H_T/X\times_S T/T}$. When $S=\Spec\kappa$, $\kappa$ a field, the Quot scheme $\Quot^{P,L}_{H/X/\kappa}$ is a fine moduli space whose $\kappa$-valued points parametrize equivalence classes of quotients $H\surj F$ with $F$ having Hilbert polynomial $P$. In the case that $S$ is a general scheme, by the base-change property, $\Quot^{P,L}_{H/X/S}$ can be viewed as a family of Quot schemes parametrized by $S$: for any point $s\in S$, the fiber of $\Quot^{P,L}_{H/X/S}$ over $s$ is $\Quot^{P,L_s}_{H_s/X_s/\kappa(s)}$, where $\kappa(s)$ is the residue field of $s$. This is why $\Quot^{P,L}_{H/X/S}$ is also called a relative Quot scheme.

We now fix some notations. Let $\Bbbk$ be a fixed algebraically closed field and $\mathbb{P}^1$ the projective line over $\Bbbk$. For any $\Bbbk$-scheme $X$, we write $\mathbb{P}^1_X:=\mathbb{P}^1\times_\Bbbk X$ and denote the projection $\mathbb{P}^1_X\to X$ by $\pi_X$, or simply $\pi$. For any point $x\in X$, we denote by $\mathbb{P}^1_x$ the fiber $\pi_X^{-1}(x)=\mathbb{P}^1\times_\Bbbk\kappa(x)$ over $x$. For any coherent sheaf $F$ over $\mathbb{P}^1_X$, we write $F(m):=F\otimes p^*\mathcal{O}_{\mathbb{P}^1}(m)$, where $p$ denotes the projection $\mathbb{P}^1_X\to \mathbb{P}^1$.

For any morphism $f:Y\to X$ of $\Bbbk$-schemes, we denote by $\bar f: \mathbb{P}^1_X\to \mathbb{P}^1_Y$ the induced morphism: $\bar f(t,y)=(t,f(y))$. It is the same as the projection $\mathbb{P}^1_Y=\mathbb{P}^1_X\times_X Y\to \mathbb{P}^1_X$. If $Z\subset X$ is a subscheme and $F$ a coherent sheaf on $\mathbb{P}^1_X$, we denote by $F_Z$ the restriction of $F$ on $\mathbb{P}^1_Z$. In particular, if $x\in X$ is a point, then $F_x$ is the restriction of $F$ on $\mathbb{P}^1_x$.

In this paper, we will only encounter Quot schemes of the form $\Quot_{H/\mathbb{P}^1_S/S}^{P,\mathcal{O}_{\mathbb{P}^1_S}(1)}$, where $S$ is either $\Spec\Bbbk$ or a variety over $\Bbbk$. For any coherent sheaf $F$ on $\mathbb{P}^1$, its Hilbert polynomial has the form $P(t)=r(t+1)+d$, where $r=\rank F$ and $d=\deg F$. So we fix the notation:
\[\Quot^{r,d}_{H/\mathbb{P}^1_S/S}:=\Quot^{r(t+1)+d,\mathcal{O}_{\mathbb{P}^1_S}(1)}_{H/\mathbb{P}^1_S/S}.\]

Next we prove some basic properties of the functor $\EExt^1_X(-,\mathcal{O}_X)$. For any scheme $X$ and any sheaf $F$ over $X$, we put
\[
F^\vee:=\HHom_X(F,\mathcal{O}_X),\quad F^\e:=\EExt^1_X(F,\mathcal{O}_X).
\]
\begin{prop}\label{Pullback-Ext}
Let $0\to E_1\to E_0\to F\to 0$ be a short exact sequence of coherent sheaves on a scheme $X$ with
$E_0$ and $E_1$ locally free. For any scheme $Y$ and any morphism $f: Y\to X$, we have a canonical identification
\[
f^*(F^\varepsilon)=(f^*F)^\varepsilon
\]
\end{prop}
\begin{proof}
Since $E_0$ and $E_1$ are both locally free, we have canonical
identifications
\[
f^*(E_i^\vee) =(f^*E_i)^\vee,\quad i=0,1.
\]
Thus we obtain a commutative diagram
\[
\CD{ %
(f^*E_0)^\vee \ar[r]\ar@{=}[d] & (f^*E_1)^\vee\ar[r]\ar@{=}[d]& (f^*F)^\varepsilon \ar@{::}[d]\ar[r] & 0 \\
f^*(E_0^\vee)\ar[r] & f^*(E_1^\vee)\ar[r] & f^*(F^\varepsilon) \ar[r] & 0
}%
\]
where the first row is obtained by applying $f^*$ to the exact sequence first and taking dual
second, and the second row is obtained by taking dual of the original exact sequence first and
applying $f^*$ second.
So we have another canonical identification $f^*(F^\varepsilon)=(f^*F)^\varepsilon$ because they
are both quotients of $f^*E_0^\vee$ by the image of $f^*E_1^\vee$.
\end{proof}
\begin{rem}
For any locally free sheaf $E$, since we have a canonical identification $f^*(E^\vee)=(f^*E)^\vee$, we will write $f^*E^\vee$ for both. In the same manner, for any coherent sheaf $F$, when we have a canonical identification $f^*(F^\e)=(f^*F)^\e$, we will write $f^*F^\e$ for both.
\end{rem}

\begin{prop}\label{FlatTorsion}
Suppose $X$ is a noetherian $\Bbbk$-scheme, and $\mathcal{T}$ is a torsion coherent sheaf on $\mathbb{P}^1_X$, flat over $X$ with
relative degree $d$. Then
\begin{enumerate}
\item $\mathcal{T}^\varepsilon$ is also torsion and flat over $X$ with relative degree $d$;
\item we have a canonical isomorphism $\mathcal{T}^{\varepsilon\varepsilon}= \mathcal{T}$.
\end{enumerate}
\end{prop}
\begin{proof}
Write $\mathcal{T}$ as a quotient of a locally free sheaf $\mathcal{E}_0$ and let $\mathcal{E}_1$ be the kernel:
\begin{equation}\label{SES1}
0\to \mathcal{E}_1\to \mathcal{E}_0\to \mathcal{T}\to 0.
\end{equation}
Then $\mathcal{E}_1$ is also locally free. Since $\rank\mathcal{T}=0$, we have $\rank\mathcal{E}_0=\rank\mathcal{E}_1$. Suppose $\mathcal{E}_i$ is of relative degree $d_i$ over $X$, $i=0,1$. Then $d=d_0-d_1$.

(1) Dualizing the exact sequence (\ref{SES1}), we obtain another exact sequence
\begin{equation}\label{SES2}
0\to \mathcal{E}_0^\vee \to \mathcal{E}_1^\vee \to \mathcal{T}^\varepsilon \to 0
\end{equation}
To show that $\mathcal{T}^\varepsilon$ is flat over $X$, we only need to show, by Lemma 2.1.4 in \cite{Huy}, that for every point $x\in X$, the map $(\mathcal{E}_0)^\vee_x\to
(\mathcal{E}_1)^\vee_x$ obtained by restricting the injective map $\mathcal{E}_0^\vee \to
\mathcal{E}_1^\vee$ to the fiber $\mathbb{P}^1_x$ is injective. We first restrict the exact sequence
(\ref{SES1}) to the fiber $\mathbb{P}^1_x$. Since $\mathcal{T}$ is flat, we obtain an exact sequence
\[
0\to (\mathcal{E}_1)_x \to (\mathcal{E}_0)_x \to \mathcal{T}_x \to 0
\]
Taking dual:
\[
0=(\mathcal{T}_x)^\vee\to (\mathcal{E}_0)_x^\vee \to (\mathcal{E}_1)_x^\vee \to \mathcal{T}_x^\varepsilon \to 0
\]
since $\mathcal{T}_x$ is torsion. So $(\mathcal{E}_0)^\vee_x\to
(\mathcal{E}_1)^\vee_x$ is injective for all $x\in X$. Hence $\mathcal{T}^\varepsilon$ is flat over
$X$. Since $\mathcal{E}_i^\vee$ are both locally free of the same rank and has relative degree $-d_i$, $i=0,1$, $\mathcal{T}^\e$ is torsion on $\mathbb{P}^1_X$ with relative
degree $(-d_1)-(-d_0)=d$ over $X$.

(2) We have a commutative diagram
\[
\CD{ %
0\ar[r] & \mathcal{E}_1 \ar[r]\ar@{=}[d] & \mathcal{E}_0 \ar[r]\ar@{=}[d] & \mathcal{T} \ar[r]
\ar@{::}[d] & 0 \\
0\ar[r] & \mathcal{E}_1^{\vee\vee}\ar[r] & \mathcal{E}_0^{\vee\vee}\ar[r] & \mathcal{T}^{\e\e} \ar[r] &
0
}%
\]
where the second row is an exact sequence obtained by dualizing the exact sequence (\ref{SES2}).
Thus we see that there is a canonical isomorphism $\mathcal{T}^{\e\e}= \mathcal{T}$.
\end{proof}

Lastly, we review a slightly extended result of Mumford on flattening stratifications, as stated in the following proposition. Parts (1)-(3) are due to Mumford, and part (4) is added for the purpose of this paper.
\begin{prop}\label{p8}
Let $X$ be a noetherian scheme, $F$ a coherent sheaf on $X$, and $R=\{\dim_{\kappa(x)}F|_x$ $|\,
x\in X\}$. For each $r\in R$, define $C_r:=\{x\in X\,|\,\dim_{\kappa(x)}F|_x=r\}$. Then
\begin{enumerate}
\item $R$ is a finite set. The sets $C_r$ are disjoint and their union is $|X|$.

\item For each $x\in C_r$, there is an affine open set $U_x\ni x$ together with an exact sequence of sheaves on $U_x$
\[
\mathcal{O}_{U_x}^s\stackrel{\psi}{\to} \mathcal{O}_{U_x}^r\stackrel{\varphi}{\to} F|_{U_x}\to 0
\]
such that the ideal $J_{r,U_x}$ generated by all entries of $\psi$ defines a subscheme $Y_{r,U_x}$ of $U_x$ whose support is exactly $C_r\cap U_x$. The subschemes $Y_{r,U_x}$ for all $x\in C_r$ patch together to form a subscheme of $X$, whose support is exactly $C_r$. Denote this subscheme by $X_r$.

\item The subschemes $X_r$ satisfy the following universal property: for any
morphism $f: Y\to X$ of noetherian schemes, the pullback $f^*F$ is locally free over $Y$ of rank
$r$ if and only if $f$ factors through the inclusion $X_r\embed X$. In particular, $F|_{X_r}$ is locally free of rank $r$
for each $r\in R$.

\item Assume in addition that $E_1\stackrel{\rho}{\to} E_0\to F\to 0$ is a locally free presentation of $F$
with $\rank E_0=l$. For $j\geq 1$, let $I_j$ be the sheaf image of the homomorphism
$\HHom(\bigwedge^{j}E_1,$ $\bigwedge^{j}E_0)^\vee \to \mathcal{O}_X$ induced by $\bigwedge^{j}\rho:
\bigwedge^{j}E_1\to \bigwedge^{j}E_0$, and let $Z_{j}$ be the closed subscheme defined by the sheaf
of ideals $I_j$. Then $Z_j\subset Z_{j+1}$ and $Z_{j+1}\setminus Z_{j}=X_{l-j}$ as schemes for
$j\geq 1$.
\end{enumerate}
\end{prop}

\begin{proof}
(1), (2) and (3) are proved in \cite{Mumford}, Lecture 8, pg 55, Case $1^\circ$. We now prove (4). Suppose $F$ has a locally free presentation $E_1\stackrel{\rho}{\to}E_0\to F\to 0$ on $X$.
Then we have $|Z_j|=\{x\in X\,|\,\dim_{\kappa(x)}F|_x\geq l-j+1\}$, which obviously satisfy
$|Z_j|\subset |Z_{j+1}|$ and $|Z_{j+1}|\setminus |Z_j|=|X_{l-j}|$. At any point $x\in X_{l-j}$, there
exists an affine open neighborhood $V_x$ of $x$, possibly smaller than $U_x$, such that
\[
E_1|_{V_x}\stackrel{\rho}{\to} E_0|_{V_x}\to F|_{V_x}\to 0
\]
is a free presentation of $F$ on $V_x$. Thus the ideal $I_{j+1}(V_x)$, which is generated by all
$(j+1)\times (j+1)$ minors of $\rho$, is the same as the ideal $J_{x}$ on $V_x$, by the Fitting
ideal lemma. So $(Z_{j+1}\setminus Z_j)\cap V_x=X_{l-j}\cap V_x$ as schemes for all $x\in X_{l-j}$.
Therefore $Z_{j+1}\setminus Z_j=X_{l-j}$ as schemes.
\end{proof}

Before we proceed to the next section, let us mention a few more conventions. Let $E$ and $F$ be two locally free sheaves on $X$. We say $E$ is a \emph{subbundle} of $F$, if there is an injective homomorphism $E\to F$ of sheaves and the quotient $F/E$ is locally free. In this case, we say the homomorphism $E\to F$ is a \emph{subbundle map}.

For any locally free sheaf $E$ on a scheme $X$, the \emph{projective bundle} on $X$ associated to $E$ is
\[
\mathbb{P}(E):=\PProj\Sym (E^\vee).
\]
It is equipped with the universal quotient line bundle $f^*E^\vee\surj \mathcal{O}_{\mathbb{P}(E)}(1)$, where $f:\mathbb{P}(E)\to X$ is the structure morphism. Its dual is the universal subline bundle $\mathcal{O}_{\mathbb{P}(E)}(-1)\to f^*E$. For any coherent sheaf $F$ on $\mathbb{P}(E)$, we write $F(m):=F\otimes \mathcal{O}_{\mathbb{P}(E)}(m)$.

\section{The Quot-scheme compactification and its boundary}
From now on, we work over an algebraically closed field $\Bbbk$. Let $V$ be a vector space of dimension $n$ over $\Bbbk$, and $k<n$ a fixed positive integer. The Grassmannian $\Gr:=\Gr(k,V)$ of $k$-dimensional subspaces of $V$ comes equipped with a universal quotient $V_{\Gr}\surj W$ of locally free sheaves on $\Gr$. The \emph{degree} of a morphism $f:\mathbb{P}^1\to \Gr$ is equal to $\deg f^*W$, which coincides with the usual degree of the map $\mathbb{P}^1\to \mathbb{P}(\bigwedge^k V)$ obtained by composing $f$ with the Pl\"{u}cker embedding $\Gr\embed \mathbb{P}(\bigwedge^k V)$. For $d\geq 0$, let $\Mor_d(\mathbb{P}^1,\Gr)$ denote the set of all degree $d$ morphisms from $\mathbb{P}^1$ to $\Gr$. The set $\Mor_d(\mathbb{P}^1,\Gr)$ has a natural structure of a (quasi-projective) variety, which can be realized as an open subscheme of the Hilbert scheme of $\mathbb{P}^1_\Gr$ corresponding to graphs.

There is a one-to-one correspondence between the set $\Mor_d(\mathbb{P}^1,\Gr)$ and the set of all equivalence classes of quotients $V_{\mathbb{P}^1}\surj F$ where $F$ is locally free, which is set up by associating each degree $d$ morphism to the pullback of the universal quotient on $\Gr$. This identifies the space $\Mor_d(\mathbb{P}^1,\Gr)$ with the open locus of locally free quotients in the Quot scheme $\Quot^{n-k,d}_{V_{\mathbb{P}^1}/\mathbb{P}^1/\Bbbk}$. This Quot scheme is an irreducible nonsingular projective variety of dimension $nd+k(n-k)$ \cite{Str87}, hence it provides a smooth compactification of $\Mor_d(\mathbb{P}^1,\Gr)$. For short notations, we put
\[
Q_d:=\Quot^{n-k,d}_{V_{\mathbb{P}^1}/\mathbb{P}^1/\Bbbk},\quad \mathring{Q}_d:=\Mor_d(\mathbb{P}^1,\Gr)\subset Q_d
\]
The Quot scheme $Q_d$ is equipped with a universal short exact sequence on $\mathbb{P}^1_{Q_d}$:
\begin{equation}\label{UES-Qd}
0\to \mathcal{E}_d\to V_{\mathbb{P}^1_{Q_d}}\to \mathcal{F}_d \to 0
\end{equation}
where $\mathcal{F}_d$ is flat over $Q_d$ of rank $n-k$ and relative degree $d$, and $\mathcal{E}_d$ is locally free of rank $k$ and relative degree $-d$.

For any point $q\in Q_d$, the pullback of the universal exact sequence to $\mathbb{P}^1_q$ is also an exact sequence:
\[
0\to \mathcal{E}_{d,q}\to V_{\mathbb{P}^1_q}\to \mathcal{F}_{d,q}\to 0
\]
The boundary $Q_d\setminus \mathring{Q}_d$ consists of the points $q$ with $\mathcal{F}_{d,q}$ having torsion. For any coherent sheaf $F$, we denote by $F_\tor$ the torsion submodule of $F$. We define a chain of closed subsets
\[
C_{d,0}\subset C_{d,1}\subset \cdots \subset C_{d,d-1}=Q_d\setminus\mathring{Q}_d
\]
by setting $C_{d,r}=\{q\in Q_d\,|\,\deg (\mathcal{F}_{d,q})_\tor\geq d-r\}.$

The degree of the torsion part $(\mathcal{F}_{d,q})_{\tor}$ of $\mathcal{F}_{d,q}$ can be determined by the rank of certain linear maps. Fix a point $q\in Q_d$ and set $E:=\mathcal{E}_{d,q}$, $F:=\mathcal{F}_{d,q}$, $T:=(\mathcal{F}_{d,q})_\tor$, and $V:=V_{\mathbb{P}^1_q}$. For $r\geq 0$, let
\[
\rho_{q,r}: \Hom(V,\mathcal{O}(r))\to \Hom(E,\mathcal{O}(r))
\]
be the map obtained by applying $\Hom(-,\mathcal{O}(r))$ to $E\to V$. Here
$\mathcal{O}(r):=\mathcal{O}_{\mathbb{P}^1_q}(r)$. Then we have
\begin{prop}\label{p1}
\[
\rank \rho_{q,r}\left\{\begin{array}{ll}
=k(r+1)+d-\deg T,& \text{ if }\ r\geq d-\deg T-1,\\
\geq (k+1)(r+1),& \text{ if }\ 0\leq r\leq d-\deg T-1.
\end{array}\right.
\]
\end{prop}
\begin{proof}
We have an isomorphism $F\simeq T\oplus F_1$, where $F_1=F/T$ is a locally free sheaf
of rank $n-k$ and degree $d-\deg T$. Applying $\Hom(-,\mathcal{O}(r))$ to the exact sequence, one
gets an exact sequence
\[
0\to \Hom(F,\mathcal{O}(r))\to \Hom(V,\mathcal{O}(r))\stackrel{\rho_{_{q,r}}}{\to}
\Hom(E,\mathcal{O}(r))
\]
Thus
\[
\begin{split}
\rank \rho_{q,r}&=\dim \Hom(V,\mathcal{O}(r))-\dim \Hom(F,\mathcal{O}(r))\\
& =\dim H^0(V^\vee(r))-\dim \Hom(F_1,\mathcal{O}(r)) \\
& =n(r+1)-\dim H^0(F_1^\vee(r))
\end{split}
\]
We have an isomorphism $F_1\simeq\bigoplus_{i=1}^{n-k}\mathcal{O}(d_i)$, where $d_i\geq 0$ and
$\sum_{i=1}^{n-k}d_i=d-\deg T$. Hence
\[
H^0(F_1^\vee(r))\simeq H^0(\bigoplus_{i=1}^{n-k}\mathcal{O}(r-d_i))=\bigoplus_{i=1}^{n-k}H^0(\mathcal{O}(r-d_i))
\]
and
\[
\dim H^0(F_1^\vee(r))=\sum_{i=1}^{n-k}\max\{r-d_i+1,0\}.
\]
If $r\geq d-\deg T-1$, then $r-d_i+1\geq 0$ for all $i$, hence
\[
\dim H^0(F_1^\vee(r))=\sum_{i=1}^{n-k}(r-d_i+1)=(n-k)(r+1)+\sum_{i=1}^{n-k}d_i=(n-k)(r+1)-(d-\deg
T)
\]
Therefore
\[
\rank \rho_{q,r}=n(r+1)-(n-k)(r+1)+d-\deg T=k(r+1)+d-\deg T.
\]
If $0\leq r\leq d-\deg T-1$, then we have two cases:

(i) $r-d_i+1\geq 0$ for all $i$. Then
\[
\begin{split}
\dim H^0(F_1^\vee(r))&=\sum_{i=1}^{n-k}(r-d_i+1)=(n-k)(r+1)-(d-\deg T)\\
&\leq (n-k)(r+1)-(r+1)=(n-k-1)(r+1)
\end{split}
\]

(ii) $r-d_i+1<0$ for some $i$. Without loss of generality, we assume $r-d_1+1<0$. Then
\[
\dim H^0(F_1^\vee(r))=\sum_{i=2}^{n-k}\max\{r-d_i+1,0\}\leq \sum_{i=2}^{n-k}(r+1)=(n-k-1)(r+1).
\]
In either case, we have $\dim H^0(F_1^\vee(r))\leq (n-k-1)(r+1)$, so
\[
\rank \rho_{q,r}\geq n(r+1)-(n-k-1)(r+1)=(k+1)(r+1).
\]
\end{proof}

Note that the two ``if'' conditions in the proposition are not mutually exclusive, and this will lead to some interesting results.
\begin{cor}\label{c3}
If $\rank \rho_{q,i}\leq k(i+1)+i$ for some $i\geq 0$, then $\deg T\geq d-i$ and
\[
\rank \rho_{q,l}-\rank \rho_{q,i}=k(l-i),\quad \text{for all } l\geq i-1.
\]
\end{cor}
\begin{proof}
Suppose $\rank\rho_{q,i}\leq k(i+1)+i<(k+1)(i+1)$ for some $i\geq 0$. By Proposition \ref{p1}, we must have $i\geq
d-\deg T$, hence $\deg T\geq d-i$. We now have $i-1\geq d-\deg T-1$. Thus for $l\geq i-1$, we obtain, by Proposition \ref{p1} again, that
\[
\rank\rho_{q,l}-\rank\rho_{q,i}=(k(l+1)+d-\deg T)-(k(i+1)+d-\deg T)=k(l-i).
\]
\end{proof}

\begin{prop}\label{c1}
\begin{enumerate}
\item If $\deg T=d-r$ (resp. $\geq d-r$), then $\rank\rho_{q,m}=k(m+1)+r$ (resp. $\leq
k(m+1)+r$) for all $m\geq r$.
\item If $\rank\rho_{q,m}=k(m+1)+r$ (resp. $\leq
k(m+1)+r$) for some $m\geq r$, then $\deg T=d-r$ (resp. $\geq d-r$).
\end{enumerate}
\end{prop}
\begin{proof}
(1) Suppose $\deg T=d-r$ (resp. $\geq d-r$). Then $r=d-\deg T$ (resp. $\geq d-\deg T$), and by Proposition
\ref{p1}, we have $\rank \rho_{q,r}=k(r+1)+d-\deg T=k(r+1)+r$ (resp. $\leq k(r+1)+r$). Then by Corollary \ref{c3}, for any $m\geq r$, $\rank
\rho_{q,m}=\rank\rho_{q,r}+k(m-r)=k(m+1)+r$ (resp. $\leq k(m+1)+r$).

(2) Suppose $\rank\rho_{q,m}=k(m+1)+r$ (resp. $\leq k(m+1)+r$) for some $m\geq r$. Then $\rank\rho_{q,m}\leq k(m+1)+m$. Let $m'=m-1$. By Corollary \ref{c3}, $\rank\rho_{q,m'}-\rank\rho_{q,m}=-k$. So $\rank\rho_{q,m'}=\rank\rho_{q,m}-k=k(m+1)+r-k=k(m'+1)+r$ (resp. $\leq k(m'+1)+r$). If $m'\geq r$, then we are back in the starting situation ($m'$ playing the role of $m$), so we can repeat this process, and eventually we must have $\rank\rho_{q,r}=k(r+1)+r$ (resp. $\leq k(r+1)+r$). By Corollary \ref{c3}, we have $\deg T\geq d-r$, or $r\geq d-\deg T$. Then by Proposition \ref{p1}, $d-\deg T=\rank\rho_{q,m}-k(r+1)=r$ (resp. $\leq r$), and hence $\deg T=d-r$ (resp. $\geq d-r$).
\end{proof}

For each integer $m$, we set $\widehat m:=k(m+1)$. Proposition \ref{c1} implies that for $m\geq r$, $q\in C_{d,r}\Longleftrightarrow $ the induced homomorphism of the exterior powers \[\bigwedge^{\widehat m+r+1}\rho_{q,m}:\bigwedge^{\widehat m+r+1}\Hom(V,\mathcal{O}(r))\to \bigwedge^{\widehat m+r+1}\Hom(E,\mathcal{O}(r))\] is a zero homomorphism. This suggests a way to endow each $C_{d,r}$ with a structure of determinantal subschemes. Denote by $\pi: \mathbb{P}^1_{Q_d}\to Q_d$ the projection map. Let
\[
\rho_{d,m}: \pi_*V_{\mathbb{P}^1_{Q_d}}^\vee(m)\to \pi_*\mathcal{E}_d^\vee(m)
\]
be the $\mathcal{O}_{Q_d}$-homomorphism obtained by applying $\HHom(-,\mathcal{O}_{\mathbb{P}^1_{Q_d}}(m))$ to the monomorphism $\mathcal{E}_d\to
V_{\mathbb{P}^1_{Q_d}}$ from the universal exact sequence first and $\pi_*$ second.

For each $l\geq 1$, the exterior power of $\rho_{d,m}$
\[
\bigwedge^{l}\rho_{d,m}: \bigwedge^{l}\pi_*V_{\mathbb{P}^1_{Q_d}}^\vee(m)\to
\bigwedge^{l}\pi_*\mathcal{E}_d^\vee(m)
\]
is a section of $\HHom\big(\bigwedge^{l}\pi_*V_{\mathbb{P}^1_{Q_d}}^\vee(m),
\bigwedge^{l}\pi_*\mathcal{E}_d^\vee(m)\big)$. By Proposition \ref{c1}, for $m\geq d$, $\bigwedge^{\widehat m+r+1}\rho_{d,m}$ is nowhere vanishing on $Q_d$ for $r\leq -1$, but vanishes precisely on $C_{d,r}$ for $r\geq 0$.

The section $\bigwedge^{\widehat m+r+1}\rho_{d,m}$ induces a homomorphism
\[
\HHom\bigg(\bigwedge^{\widehat m+r+1}\pi_*V_{\mathbb{P}^1_{Q_d}}^\vee(m), \bigwedge^{\widehat
m+r+1}\pi_*\mathcal{E}_d^\vee(m)\bigg)^\vee\to \mathcal{O}_{Q_d}.
\]
The image, which we denote by $I_{d,r,m}$, is an ideal sheaf. For $m\geq d$, each $I_{d,r,m}$ defines a subscheme of $Q_d$ whose support is $C_{d,r}$. A natural question is: is the scheme structure independent of $m$? We will show that the answer is yes if $m\gg 0$.

Dualizing the universal exact sequence (\ref{UES-Qd}), one obtains a long exact sequence
\begin{equation}\label{a2}
0 \to \mathcal{F}_d^\vee \to V^\vee_{\mathbb{P}^1_{Q_d}} \to \mathcal{E}_d^\vee \to \mathcal{F}_d^\varepsilon \to 0
\end{equation}
We need to study the sheaf $\mathcal{F}_d^\e$. By Proposition \ref{Pullback-Ext}, for any morphism $f: X\to \mathbb{P}^1_{Q_d}$, we can write $f^*\mathcal{F}_d^\e$ for both $f^*(\mathcal{F}_d^\e)$ and $(f^*\mathcal{F}_d)^\e$. In particular, for any subscheme $Z\subset Q_d$, we write $(\mathcal{F}_{d})_Z^\e$ for $(\mathcal{F}_d^\e)_Z$ and $((\mathcal{F}_d)_Z)^\e$.
Some facts are quick: $(\mathcal{F}_d)^\e_{\mathring{Q}_d}=((\mathcal{F}_d)_{\mathring{Q}_d})^\e=0$ since $(\mathcal{F}_d)_{\mathring{Q}_d}$ is locally free, and for any point $q\in Q_d$, $\mathcal{F}_{d,q}^\e\simeq ((\mathcal{F}_{d,q})_\tor)^\e$ since $\mathcal{F}_{d,q}$ splits as a direct sum of its torsion part and locally free part. Then by Proposition \ref{FlatTorsion}, we see that $\mathcal{F}_{d,q}^\e$ is torsion and $\deg \mathcal{F}_{d,q}^\e=\deg (\mathcal{F}_{d,q})_\tor$. Thus $\mathcal{F}_d^\e$ is a torsion sheaf and $\Supp\mathcal{F}_d^\e\subset\mathbb{P}^1_{Q_d}\setminus \mathbb{P}^1_{\mathring{Q}_d}$.
\begin{lemma}
For any closed  point $q\in Q_d$, there is an open neighborhood $U$ of $q$ such that
$\mathcal{F}_d^\varepsilon|_{\mathbb{P}^1_U}\simeq
\mathcal{F}_d^\varepsilon(m)|_{\mathbb{P}^1_U}$ for all $m\geq 0$.
\end{lemma}
\begin{proof}
Let $q\in Q_d$ be a closed point and $i_q:\mathbb{P}^1\to \mathbb{P}^1_{Q_d}$ the inclusion map defined by $i_q(t)=(t,q)$. So the image of $i_q$ is exactly $\mathbb{P}^1_q$. Since the sheaf $\mathcal{F}_{d,q}^\varepsilon$ is torsion on $\mathbb{P}^1_q$, the support of
$\mathcal{F}_{d,q}^\varepsilon$ consists of only finitely many points of $\mathbb{P}^1_q$. Therefore, $\Supp\mathcal{F}_d^\varepsilon$ is a proper closed subset
of $\mathbb{P}^1_{Q_d}$. Choose a point $t\in \mathbb{P}^1$ so that
$i_q(t)\in\mathbb{P}^1_q\setminus\Supp\mathcal{F}_{d,q}^\varepsilon$ and let $U=Q_d\setminus
\pi(p^{-1}(t) \cap \Supp\mathcal{F}_d^\varepsilon)$, where $p:\mathbb{P}^1_{Q_d}\to \mathbb{P}^1$ is the projection. Obviously, $q\in U$. Now choose an exact
sequence on $\mathbb{P}^1$:
\[
0\to \mathcal{O}_{\mathbb{P}^1}\to \mathcal{O}_{\mathbb{P}^1}(m)\to T \to 0
\]
such that $\Supp T=\{t\}$. Pulling the sequence back to $\mathbb{P}^1_U$ (which preserves exactness) and tensoring with
$\mathcal{F}_d^\varepsilon|_{\mathbb{P}^1_U}$, we get a long exact sequence
\[
\TTor_1(\mathcal{F}_d^\varepsilon|_{\mathbb{P}^1_U},p^*T|_{\mathbb{P}^1_U})\to
\mathcal{F}_d^\varepsilon|_{\mathbb{P}^1_U}\to \mathcal{F}_d^\varepsilon(m)|_{\mathbb{P}^1_U}\to
\mathcal{F}_d^\varepsilon|_{\mathbb{P}^1_U}\otimes p^*T|_{\mathbb{P}^1_U}\to 0
\]
Note that $\Supp\mathcal{F}_d^\varepsilon$ and $\Supp p^*T$ do not intersect on $\mathbb{P}^1_U$, hence $\TTor_1(\mathcal{F}_d^\varepsilon|_{\mathbb{P}^1_U},p^*T|_{\mathbb{P}^1_U})=0$ and $\mathcal{F}_d^\varepsilon|_{\mathbb{P}^1_U}\otimes p^*T|_{\mathbb{P}^1_U}=0$, which gives rise to an isomorphism $\mathcal{F}_d^\varepsilon\simeq \mathcal{F}_d^\varepsilon(m)$ on $\mathbb{P}^1_U$. Obviously, this open set $U$ works for all $m\geq 0$.
\end{proof}

\begin{prop}\label{c2}
\begin{enumerate}
\item For any closed point $q\in Q_d$, there is an open neighborhood $U$ of $q$ such that
$\pi_*\mathcal{F}_d^\varepsilon(m)|_U\simeq\pi_*\mathcal{F}_d^\varepsilon(l)|_U$, for all $m,l\geq
0$.

\item There is an integer $m_d\geq d$ (independent of $r$) such that $I_{d,r,m}=I_{d,r,m_d}$ as subsheaves of
$\mathcal{O}_{Q_d}$ for all $m\geq m_d$ and for any $0\leq r\leq d-1$.
\end{enumerate}
\end{prop}

\begin{proof}
(1) By the above lemma, we deduce that there exists $U$ such that
$\mathcal{F}_d^\varepsilon(m)|_{\mathbb{P}^1_U}\simeq
\mathcal{F}_d^\varepsilon(l)|_{\mathbb{P}^1_U}$ for arbitrary $m,l$. Then
\[
\pi_{U*}(\mathcal{F}_d^\varepsilon(m)|_{\mathbb{P}^1_U})\simeq
\pi_{U*}(\mathcal{F}_d^\varepsilon(l)|_{\mathbb{P}^1_U})
\]
The statement follows from the facts that
\[
\pi_*\mathcal{F}_d^\varepsilon(m)|_U=\pi_{U*}(\mathcal{F}_d^\varepsilon(m)|_{\mathbb{P}^1_U})\quad\text{ and
}\quad\pi_*\mathcal{F}_d^\varepsilon(l)|_U=\pi_{U*}(\mathcal{F}_d^\varepsilon(l)|_{\mathbb{P}^1_U}).
\]

(2) We break up the exact sequence (\ref{a2}) into two short exact sequences:
\[
0\to \mathcal{F}_d^\vee\to V^\vee_{\mathbb{P}^1_{Q_d}}\to \mathcal{H}\to 0, \quad 0\to
\mathcal{H}\to \mathcal{E}_d^\vee\to \mathcal{F}_d^\varepsilon\to 0
\]
Twisting the two sequences with $\mathcal{O}_{\mathbb{P}^1_{Q_d}}(m)$ (which preserves exactness) and
applying $\pi_*$, we get another two exact sequences
\[
\begin{split}
& 0\to \pi_*\mathcal{F}_d^\vee(m) \to \pi_*V^\vee_{\mathbb{P}^1_{Q_d}}(m)\to
\pi_*\mathcal{H}(m)\to R^1\pi_*\mathcal{F}_d^\vee(m) \\
& 0\to \pi_*\mathcal{H}(m)\to \pi_*\mathcal{E}_d^\vee(m)\to \pi_*\mathcal{F}_d^\varepsilon(m)\to
R^1\pi_*\mathcal{H}(m)
\end{split}
\]
There is an integer $m_d\geq d$ such that
$R^1\pi_*\mathcal{F}_d^\vee(m)=0=R^1\pi_*\mathcal{H}(m)$ for all $m\geq m_d$. Hence for $m\geq m_d$, the above two sequences
join together into one long exact sequence:
\begin{equation}\label{mainExactSequence}
0\to \pi_*\mathcal{F}_d^\vee(m) \to \pi_*V^\vee_{\mathbb{P}^1_{Q_d}}(m)\stackrel{\rho_{_{d,m}}}{\to} \pi_*\mathcal{E}_d^\vee(m)\to
\pi_*\mathcal{F}_d^\varepsilon(m)\to 0
\end{equation}
We see that $\pi_*V^\vee_{\mathbb{P}^1_{Q_d}}(m)=V_{Q_d}^\vee\otimes
H^0(\mathbb{P}^1,\mathcal{O}(m))$ is a free $\mathcal{O}_{Q_d}$-module, and $\pi_*\mathcal{E}_d^\vee(m)$ is locally free (of rank $(\widehat m+d)$).

Now fix an $m> m_d$. We can choose an affine open cover $U_i$ of $Q_d$ such that $\pi_*\mathcal{F}_d^\e(m)|_{U_i}\simeq \pi_*\mathcal{F}_d^\e(m_d)|_{U_i}$ on each $U_i$. We may also assume that $\pi_*\mathcal{E}_d^\vee(m)$ and $\pi_*\mathcal{E}_d^\vee(m_d)$ are both free on each $U_i$. We have exact sequences
\[
\pi_*V^\vee_{\mathbb{P}^1_{Q_d}}(l)|_{U_i}\stackrel{\rho_{_{d,l}}|_{U_i}}{\to}
\pi_*\mathcal{E}_{d}^\vee(l)|_{U_i}\to \pi_*\mathcal{F}_{d}^\varepsilon(l)|_{U_i}\to 0, \quad l=m, m_d.
\]
So the ideals $I_{d,r,l}(U_i)\subset \mathcal{O}_{Q_d}(U_i)$ are exactly the $(d-r-1)$th Fitting ideal of $\pi_*\mathcal{F}_{d}^\varepsilon(l)|_{U_i}$ for $l=m,m_d$. Since $\pi_*\mathcal{F}_d^\e(m)|_{U_i}\simeq \pi_*\mathcal{F}_d^\e(m_d)|_{U_i}$, we have $I_{d,r,m}(U_i)=I_{d,r,m_d}(U_i)$ for each $U_i$. It then follows that $I_{d,r,m}=I_{d,r,m_d}$.
\end{proof}

For $0\leq r\leq d-1$, we denote by $Z_{d,r}$ for the closed subscheme of $Q_d$
defined by the ideal $I_{d,r,m_d}$, and write $I_{Z_{d,r}}:=I_{d,r,m_d}$. Obviously, the support of $Z_{d,r}$ is the closed set $C_{d,r}$ defined earlier. The $Z_{d,r}$'s form a chain of closed subschemes:
\[
Z_{d,0}\subset Z_{d,1}\subset\cdots\subset Z_{d,d-1}=Q_d\setminus \mathring{Q}_d
\]
For convenience of notations, we set
\[
Z_{d,-1}:=\varnothing,\quad Z_{d,d}:=Q_d,\quad\text{ and }\quad  \mathring{Z}_{d,r}:=Z_{d,r}\setminus Z_{d,r-1},\
\text{ for } 0\leq r\leq d.
\]
Thus we have $\mathring{Z}_{d,0}=Z_{d,0}$ and $\mathring{Z}_{d,d}=\mathring{Q}_d$. Note that each $\mathring{Z}_{d,r}$ is an open subscheme of $Z_{d,r}$.

\begin{thm}\label{thmZ}
\begin{enumerate}
\item For $0\leq r\leq d$, $Z_{d,r}$ is the schematic zero locus of $\bigwedge^{\widehat
m+r+1}\rho_{d,m}$, and $|\mathring{Z}_{d,r}|=\{q\in Q_d\,|\,\deg
\mathcal{F}_{d,q}^\varepsilon=d-r\}$. In
particular, $\mathring{Z}_{d,r}$'s are pairwise disjoint, and
$|Q_d|=\bigsqcup_{r=0}^{d}|\mathring{Z}_{d,r}|$.

\item For $0\leq r\leq d$, the inclusion $i: \mathring{Z}_{d,r}\hookrightarrow Q_d$ has the
following \emph{universal property}: if $Y$ is a noetherian scheme and $f: Y\to Q_d$ is a morphism,
then ${\bar f}^*\mathcal{F}_d^\varepsilon$ is flat
over $Y$ with relative degree $d-r$ if and only if $f$ factors through $i$. In particular, the
sheaf $(\mathcal{F}_d)_{\mathring{Z}_{d,r}}^\e$ on $\mathbb{P}^1_{\mathring{Z}_{d,r}}$
is flat over $\mathring{Z}_{d,r}$ with relative degree $d-r$.
\end{enumerate}
\end{thm}

\begin{proof}
(1) follows easily from the fact that $\deg(\mathcal{F}_d^\e)_q=\deg(\mathcal{F}_{d,q})_\tor$. We now prove (2). For any $m\geq m_d$, we have an exact sequence:
\[
\pi_*V^\vee_{\mathbb{P}^1_{Q_d}}(m)\stackrel{\rho_{_{d,m}}}{\to} \pi_*\mathcal{E}_d^\vee(m)\to
\pi_*\mathcal{F}_d^\varepsilon(m)\to 0
\]
There is an integer $N_1$ such that, for all $m\geq N_1$,
\[
f^*\pi_*\mathcal{F}_d^\varepsilon(m)=\pi_{Y*}{\bar f}^*\mathcal{F}_d^\varepsilon(m)
\]
Suppose $f$ factor through $i$. By Proposition \ref{p8} (3) and (4), and by the definition of $Z_{d,r}$,
$f^*\pi_*\mathcal{F}_d^\varepsilon(m)$ is locally free for all $m\geq m_d$. So $\pi_{Y*}\bar
f^*\mathcal{F}_d^\varepsilon(m)$ is locally free for all $m\geq \max\{m_d,N_1\}$. It follows that ${\bar f}^*\mathcal{F}_d^\varepsilon$ is flat over $Y$. In particular, taking $Y=\mathring{Z}_{d,r}$ and $f$ to be the identity map, we see that $(\mathcal{F}_d)_{\mathring{Z}_{d,r}}^\e$ is flat over $\mathring{Z}_{d,r}$.

Now suppose $\bar f^*\mathcal{F}_d^\varepsilon$ is flat over $Y$. Then there is an integer $N_2$ such that $\pi_{Y*}\bar
f^*\mathcal{F}_d^\varepsilon(m)$ is locally free for all $m\geq N_2$. Thus $f^*\pi_*\mathcal{F}_d^\varepsilon(m)$ is
locally free for all $m\geq \max\{N_1,N_2,m_d\}$. By Proposition \ref{p8} again, $f$ factors through $i$.
\end{proof}
\begin{rem}
The above theorem says that the locally closed subschemes $Z_{d,0},\mathring{Z}_{d,1},\dots,\mathring{Z}_{d,d-1}$ and $\mathring{Q}_d$ form the \emph{flattening stratification} of $Q_d$ by the sheaf $\mathcal{F}_d^\e$.
\end{rem}

\section{More about the boundary}

For $d\geq r\geq 0$, we consider the relative Quot scheme over $Q_r$:
\[
Q_{d,r}:=\Quot^{0,d-r}_{\mathcal{E}_{r}/\mathbb{P}^1_{Q_r}/Q_r} 
\]
We denote by $\theta: Q_{d,r}\to Q_r$ the structure morphism. It is equipped with a universal exact sequence on $\mathbb{P}^1_{Q_{d,r}}$:
\begin{equation}\label{UniversalExactSequence2}
0\to \mathcal{E}_{d,r}\to \bar\theta^*\mathcal{E}_{r}\to \mathcal{T}_{d,r}\to 0
\end{equation}
where $\mathcal{T}_{d,r}$ is flat over $Q_{d,r}$ with relative degree $d-r$ and rank $0$ (i.e.,
$\mathcal{T}_{d,r}$ is torsion), and $\mathcal{E}_{d,r}$ is taken as a subsheaf of $\bar\theta^*\mathcal{E}_{r}$: $\mathcal{E}_{d,r}\subset\bar\theta^*\mathcal{E}_{r}$. It is easy to see that $\mathcal{E}_{d,r}$ is locally free of rank $k$ and of relative degree $-d$.

We now give a set-theoretic description of the points in $Q_{d,r}$. Let $q=[V_{\mathbb{P}^1}\surj F]$ be a closed point of $Q_r$ and let $E=\ker(V_{\mathbb{P}^1}\surj F)$. Obviously, $E$ is locally free of rank $k$ and degree $-r$. The fiber $\theta^{-1}(q)$ is the Quot
scheme $\Quot_{E/V_{\mathbb{P}^1}/\Bbbk}^{0,d-r}$. So any closed point in the fiber $\theta^{-1}(q)$ is represented by a
quotient $E\surj T$ where $T$ is torsion with degree $d-r$. The kernel $E':=\ker(E\surj T)$ is locally free of rank $k$ and degree $-d$.

\begin{lemma}
With $E'$ and $T$ defined as above, we have $\dim \Hom(E',T)=k(d-r)$ and $\Ext^1(E',T)=0$.
\end{lemma}
\begin{proof}
We have $E'\simeq\bigoplus_{i=1}^k\mathcal{O}(d_i)$. Then
\[
\begin{split}
\Hom(E',T)&\simeq\bigoplus_{i=1}^k\Hom(\mathcal{O}(d_i),T)\simeq\bigoplus_{i=1}^k\Hom(\mathcal{O},T)=H^0(T)^{\oplus
k} \\
\Ext^1(E',T)&\simeq\bigoplus_{i=1}^k \Ext^1(\mathcal{O}(d_i),T)\simeq\bigoplus_{i=1}^k
\Ext^1(\mathcal{O},T)=H^1(T)^{\oplus k}
\end{split}
\]
So $\dim\Hom(E',T)=k\dim H^0(T)=k(d-r)$. That $\Ext^1(E',T)=0$ follows from $H^1(T)=0$.
\end{proof}

By the deformation theory on Quot schemes (see \cite{Ser06}, Proposition 4.4.4), the above lemma implies that
\begin{prop}
The structure morphism $\theta: Q_{d,r}\to Q_r$ is smooth of relative dimension $k(d-r)$.
\end{prop}

\begin{prop}
$Q_{d,r}$ is an irreducible nonsingular projective variety of dimension $\dim Q_r+k(d-r)=nr+k(n-k+d-r)$.
\end{prop}
\begin{proof}
The nonsingularity, projectivity and dimension counting of $Q_{d,r}$ follow easily from the two facts: (i) $Q_r$ is a smooth and
projective variety of dimension $nr+k(n-k)$; (ii) $Q_{d,r}$ is smooth and projective over $Q_r$ of
relative dimension $k(d-r)$. It only remains to show the irreducibility. Since $Q_r$ is irreducible, by \cite{Sha94}, Ch. 1, Sec. 6, Theorem 8, $Q_{d,r}$ is irreducible if we can show that all the fibers of $\theta$ over the closed points of $Q_r$ are irreducible. The proof of this part is inspired by the proof of Theorem 2.1 in \cite{Str87}.

Let $[V_{\mathbb{P}^1}\surj F]$ be a closed point of $Q_r$. The fiber of $\theta$ over this point is a Quot scheme
$\Quot_{E/\mathbb{P}^1/\Bbbk}^{0,d-r}$ where $E=\ker(V_{\mathbb{P}^1}\surj F)$. Let $N$ be a vector space of dimension $d-r$, and let $W$ be the vector space
\[
W=\Hom(N_{\mathbb{P}^1}(-1),N_{\mathbb{P}^1})\times\Hom(E,N_{\mathbb{P}^1})
\]
Then $\dim W=2(d-r)^2+k(d-r)$. Let $\bar X=\Spec(\bigoplus_{i\geq 0}\Sym^iW^\vee)$ be the associated affine space. There are tautological morphisms $\mu$ and $\nu$ which fit into the diagram on
$\mathbb{P}^1_{\bar X}$:
\[
\CD{
  N_{\mathbb{P}^1_{\bar X}}(-1) \ar[r]^-\nu & N_{\mathbb{P}^1_{\bar X}}\ar@{->>}[r] & \Coker(\nu) \\
  & E_{\bar X} \ar[u]_-\mu &
}
\]
Let $X\subset \bar X$ be the open subvariety defined by the conditions: (i) $\nu$ is injective on each fiber over $X$, and
(ii) the induced map $E_X\to \Coker(\nu)$ is surjective. The sheaf $\Coker(\nu)$ is flat over $X$ with rank 0 and relative degree $d-r$. Thus the surjection $E_X\surj \Coker(\nu)$ gives a morphism $g: X\to \Quot_{E/\mathbb{P}^1/\Bbbk}^{0,d-r}$.

Next we show that $g$ is surjective. Let $[E\surj T]$ be a closed point of $\Quot_{E/\mathbb{P}^1/\Bbbk}^{0,d-r}$, and
$H:=H^0(\mathbb{P}^1,T)$. By Proposition 1.1 in \cite{Str87}, we have a natural exact sequence
\[
0\to H_{\mathbb{P}^1}(-1)\to H_{\mathbb{P}^1}\to T\to 0
\]
Applying $\Hom(E,-)$ to it, we obtain an exact sequence
\[
\cdots\to\Hom(E,H_{\mathbb{P}^1})\to \Hom(E,T)\to \Ext^1(E,H_{\mathbb{P}^1}(-1))\to\cdots
\]
We have
\[
\Ext^1(E,H_{\mathbb{P}^1}(-1))=H^1(H_{\mathbb{P}^1}(-1)\otimes E^\vee)=H^1(E^\vee(-1))\otimes
H_{\mathbb{P}^1}=0
\]
because $H^1(E^\vee(-1))=0$. So the map $\Hom(E,H_{\mathbb{P}^1})\to \Hom(E,T)$ is surjective, and
hence the quotient map $E\surj T$ factors through $H_{\mathbb{P}^1} \to T$ as
\[
\CD{
  H_{\mathbb{P}^1}(-1) \ar[r] & H_{\mathbb{P}^1} \ar[r] & T \\
  & E\ar@{..>}[u]\ar@{->>}[ur] &
}
\]
This diagram gives a point of $X$, whose image under $g$ is the point $[E\surj T]$. Therefore $g$ is surjective and hence $\Quot_{E/\mathbb{P}^1/\Bbbk}^{0,d-r}$ is irreducible since $X$ is irreducible.
\end{proof}

On $\mathbb{P}^1_{Q_{d,r}}$, we have two short exact sequence
\[
0\to \mathcal{E}_{d,r}\to \bar\theta^*\mathcal{E}_r\to \mathcal{T}_{d,r}\to 0, \qquad
0\to\bar\theta^*\mathcal{E}_r\to V_{\mathbb{P}^1_{Q_{d,r}}}\to \bar\theta^*\mathcal{F}_r\to 0
\]
The second is the pullback of the universal exact sequence of $Q_r$ by $\bar\theta$. Let $\mathcal{F}_{d,r}$ be the quotient of $V_{\mathbb{P}^1_{Q_{d,r}}}$ by $\mathcal{E}_{d,r}$ (based on the inclusions $\mathcal{E}_{d,r}\subset \bar\theta^*\mathcal{E}_r\subset\bar\theta^*V_{\mathbb{P}^1_{Q_r}}=V_{\mathbb{P}^1_{Q_{d,r}}}$). Then we can form a
commutative diagram as follows:
\begin{equation}\label{bigDiagram}
\CD{
       &                                     & 0\ar[d]                                        & 0\ar[d]                             & \\
0\ar[r]&\mathcal{E}_{d,r}\ar[r]\ar@{=}[d]&\bar\theta^*\mathcal{E}_{r}\ar[r]\ar[d]    & \mathcal{T}_{d,r}\ar[r]\ar@{..>}[d]     & 0 \\
0\ar[r]&\mathcal{E}_{d,r}\ar[r]          &V_{\mathbb{P}^1_{Q_{d,r}}}\ar[r]\ar[d]    & \mathcal{F}_{d,r}\ar[r]\ar@{..>}[d]     & 0  \\
       &                                     &\bar\theta^*\mathcal{F}_{r}\ar@{=}[r]\ar[d]& \bar\theta^*\mathcal{F}_r\ar[d]& \\
       &                                     &0                                               & 0                                   & \\
}
\end{equation}
where the dotted arrows are induced maps on quotients. All rows and the middle column are exact,
hence the last column is forced to be exact as well. Since $\mathcal{T}_{d,r}$ and
$\bar\theta^*\mathcal{F}_r$ are both flat over $Q_{d,r}$, so is $\mathcal{F}_{d,r}$. Moreover,
$\mathcal{F}_{d,r}$ has rank $n-k$ and relative degree $d$. Thus by the universal property of
$Q_d$, the exact sequence from the middle row determines a morphism
\[
\phi: Q_{d,r}\to Q_d
\]
such that the following diagram commutes
\begin{equation}\label{DiagramQdr}
\CD{
0\ar[r]&\mathcal{E}_{d,r}\ar[r]\ar@{=}[d]&V_{\mathbb{P}^1_{Q_{d,r}}}\ar[r]\ar@{=}[d]&\mathcal{F}_{d,r}\ar[r]\ar[d]^{\simeq} & 0  \\
0\ar[r]&\bar\phi^*\mathcal{E}_d\ar[r]  &\bar\phi^*V_{\mathbb{P}^1_{Q_d}}\ar[r] &\bar\phi^*\mathcal{F}_d\ar[r] &0\\
}
\end{equation}

\begin{prop}\label{IdealEquality}
For $0\leq l\leq d$, $I_{Z_{r,l}}\cdot \mathcal{O}_{Q_{d,r}}=I_{Z_{d,l}}\cdot \mathcal{O}_{Q_{d,r}}$, where
as convention we set $I_{Z_{r,l}}:=0$ for $l\geq r$.
\end{prop}
\begin{proof}
Fix an integer $m\gg 0$. By the definition of $I_{Z_{d,l}}$, we have
\[
\HHom\big(\bigwedge^{\widehat m+l+1}\pi_*V_{\mathbb{P}^1_{Q_d}}^\vee(m),\bigwedge^{\widehat
m+l+1}\pi_*\mathcal{E}_d^\vee(m)\big)^\vee\surj I_{Z_{d,l}}\subset \mathcal{O}_{Q_d}
\]
for all $m\gg 0$. Apply $\phi^*$ to obtain
\[
\phi^*\HHom\bigg(\bigwedge^{\widehat m+l+1}\pi_*V_{\mathbb{P}^1_{Q_d}}^\vee(m),\bigwedge^{\widehat
m+l+1}\pi_*\mathcal{E}_d^\vee(m)\bigg)^\vee\surj \phi^*I_{Z_{d,l}}\to \mathcal{O}_{Q_{d,r}}
\]
Thus, $I_{Z_{d,l}}\cdot \mathcal{O}_{Q_{d,r}}$, the image of $\phi^*I_{Z_{d,l}}\to
\mathcal{O}_{Q_{d,r}}$, is also the image of the induced homomorphism
\[
\phi^*\HHom\bigg(\bigwedge^{\widehat m+l+1}\pi_*V_{\mathbb{P}^1_{Q_d}}^\vee(m),\bigwedge^{\widehat
m+l+1}\pi_*\mathcal{E}_d^\vee(m)\bigg)^\vee\to\mathcal{O}_{Q_{d,r}}
\]
By the same argument, $I_{Z_{r,l}}\cdot \mathcal{O}_{Q_{d,r}}$ is the image of
\[
\theta^*\HHom\bigg(\bigwedge^{\widehat m+l+1}\pi_*V_{\mathbb{P}^1_{Q_r}}^\vee(m),\bigwedge^{\widehat
m+l+1}\pi_*\mathcal{E}_r^\vee(m)\bigg)^\vee\to \mathcal{O}_{Q_{d,r}}
\]
Note that since $m\gg 0$, we have
\[
\begin{split}
&\phi^*\HHom(\bigwedge^{\widehat m+l+1}\pi_*V_{\mathbb{P}^1_{Q_d}}^\vee(m),\bigwedge^{\widehat m+l+1}\pi_*\mathcal{E}_d^\vee(m))^\vee =\HHom\big(\bigwedge^{\widehat m+l+1}\phi^*\pi_*V_{\mathbb{P}^1_{Q_d}}^\vee(m),\bigwedge^{\widehat m+l+1}\phi^*\pi_*\mathcal{E}_d^\vee(m)\big)^\vee= \\
& \HHom\big(\bigwedge^{\widehat m+l+1}\pi_*\bar\phi^*V_{\mathbb{P}^1_{Q_d}}^\vee(m),\bigwedge^{\widehat m+l+1}\pi_*\bar\phi^*\mathcal{E}_d^\vee(m)\big)^\vee  =\HHom\big(\bigwedge^{\widehat m+l+1}\pi_*V_{\mathbb{P}^1_{Q_{d,r}}}^\vee(m),\bigwedge^{\widehat
m+l+1}\pi_*\mathcal{E}_{d,r}^\vee(m)\big)^\vee
\end{split}
\]
and
\[
\begin{split}
&\theta^*\HHom\big(\bigwedge^{\widehat m+l+1}\pi_*V_{\mathbb{P}^1_{Q_r}}^\vee(m),\bigwedge^{\widehat
m+l+1}\pi_*\mathcal{E}_r^\vee(m)\big)^\vee =\HHom\big(\bigwedge^{\widehat m+l+1}\theta^*\pi_*V_{\mathbb{P}^1_{Q_r}}^\vee(m),\bigwedge^{\widehat
m+l+1}\theta^*\pi_*\mathcal{E}_r^\vee(m)\big)^\vee= \\
&\HHom\big(\bigwedge^{\widehat m+l+1}\pi_*\bar\theta^*V_{\mathbb{P}^1_{Q_r}}^\vee(m),\bigwedge^{\widehat m+l+1}\pi_*\bar\theta^*\mathcal{E}_r^\vee(m)\big)^\vee =\HHom\big(\bigwedge^{\widehat m+l+1}\pi_*V_{\mathbb{P}^1_{Q_{d,r}}}^\vee(m),\bigwedge^{\widehat
m+l+1}\pi_*\bar\theta^*\mathcal{E}_r^\vee(m)\big)^\vee
\end{split}
\]
Dualizing the universal exact sequence (\ref{UniversalExactSequence2}) and twisting $m$ times, we
obtain a short exact sequence
\[
0\to \bar\theta^*\mathcal{E}_r^\vee(m)\to \mathcal{E}_{d,r}^\vee(m)\to \mathcal{T}_{d,r}^\e(m)\to 0
\]
Then applying $\pi_*$ we obtain
\[
0\to \pi_*\bar\theta^*\mathcal{E}_r^\vee(m)\to \pi_*\mathcal{E}_{d,r}^\vee(m)\to
\pi_*\mathcal{T}_{d,r}^\e(m)\to 0
\]
which is exact for $m\gg 0$. Moreover, by Proposition \ref{FlatTorsion}, $\mathcal{T}_{d,r}^\e$ is
flat over $Q_{d,r}$. So $\pi_*\mathcal{T}_{d,r}^\e(m)$ is locally free and so is $\pi_*\mathcal{E}_{d,r}^\vee(m)$ for $m\gg0$. So we have a subbundle homomorphism
\[
\bigwedge^{\widehat m+l+1}\pi_*\bar\theta^*\mathcal{E}_r^\vee(m)\to \bigwedge^{\widehat
m+l+1}\pi_*\mathcal{E}_{d,r}^\vee(m)
\]
and hence an induced subbundle homomorpism
\[
 \HHom\big(\bigwedge^{\widehat m+l+1}\pi_*V_{\mathbb{P}^1_{Q_{d,r}}}^\vee(m), \bigwedge^{\widehat
m+l+1}\pi_*\bar\theta^*\mathcal{E}_r^\vee(m)\big)\to \HHom\big(\bigwedge^{\widehat m+l+1}\pi_*V_{\mathbb{P}^1_{Q_{d,r}}}^\vee(m),\bigwedge^{\widehat
m+l+1}\pi_*\mathcal{E}_{d,r}^\vee(m)\big)
\]
Taking dual, we obtain a quotient bundle map
\[
\ds\HHom\big(\bigwedge^{\widehat m+l+1}\pi_*V_{\mathbb{P}^1_{Q_{d,r}}}^\vee(m),\bigwedge^{\widehat
m+l+1}\pi_*\mathcal{E}_{d,r}^\vee(m)\big)^\vee \surj \HHom\big(\bigwedge^{\widehat m+l+1}\pi_*V_{\mathbb{P}^1_{Q_{d,r}}}^\vee(m),\bigwedge^{\widehat m+l+1}\pi_*\bar\theta^*\mathcal{E}_r^\vee(m)\big)^\vee
\]
which yields a commutative diagram as follows.
\[
\CD{
  \phi^*\HHom\bigg(\ds\bigwedge^{\widehat m+l+1}\pi_*V_{\mathbb{P}^1_{Q_d}}^\vee(m),\bigwedge^{\widehat
  m+l+1}\pi_*\mathcal{E}_d^\vee(m)\bigg)^\vee \ar[r]\ar@{->>}[d]& \mathcal{O}_{Q_{d,r}}\ar@{=}[d]\\
  \theta^*\HHom\bigg(\ds\bigwedge^{\widehat m+l+1}\pi_*V_{\mathbb{P}^1_{Q_r}}^\vee(m),\bigwedge^{\widehat
  m+l+1}\pi_*\mathcal{E}_r^\vee(m)\bigg)^\vee\ar[r] & \mathcal{O}_{Q_{d,r}}
}
\]
Thus, $I_{Z_{d,l}}\cdot\mathcal{O}_{Q_{d,r}}$, the image of $\phi^*\HHom\big(\bigwedge^{\widehat
m+l+1}\pi_*V_{\mathbb{P}^1_{Q_d}}^\vee(m),\bigwedge^{\widehat
m+l+1}\pi_*\mathcal{E}_d^\vee(m)\big)^\vee$ in $\mathcal{O}_{Q_{d,r}}$, is equal to
$I_{Z_{r,l}}\cdot\mathcal{O}_{Q_{d,r}}$, the image of $\theta^*\HHom\big(\bigwedge^{\widehat
m+l+1}\pi_*V_{\mathbb{P}^1_{Q_r}}^\vee(m),\bigwedge^{\widehat m+l+1}\pi_*$
$\mathcal{E}_r^\vee(m)\big)^\vee$ in $\mathcal{O}_{Q_{d,r}}$.
\end{proof}

\begin{cor}
For $0\leq l\leq r$, the restriction of $\phi$ to the subscheme
$Q_{d,r}\times_{Q_r}Z_{r,l}=\theta^{-1}(Z_{r,l})$ of $Q_{d,r}$ factors through the inclusion
$Z_{d,l}\hookrightarrow{Q_d}$. In particular, $\phi$ factors through the inclusion
$Z_{d,r}\hookrightarrow Q_d$.
\end{cor}

We denote by $\varphi: Q_{d,r}\to Z_{d,r}$ the morphism factored out from $\phi: Q_{d,r}\to Q_d$.

\begin{prop}
The morphism $\varphi$ is surjective.
\end{prop}
\begin{proof}
Let $q\in Z_{d,r}$. Then $\deg(\mathcal{F}_{d,q})_{\rm tor}\ge d-r$. Let
$T\subset\mathcal{F}_{d,q}$ be a torsion subsheaf of degree $d-r$, let $F=\mathcal{F}_{d,q}/T$, and
let $E=\ker(V_{\mathbb{P}^1_q}\to \mathcal{F}_{d,q}\to F)$. Then we form the following commutative
diagram with exact rows and exact columns
\[
\CD{
       &                                       & 0\ar[d]                         & 0\ar[d]                             & \\
0\ar[r]&\mathcal{E}_{d,q}\ar@{..>}[r]\ar@{=}[d]&E\ar@{..>}[r]\ar[d]    & T\ar[r]\ar[d] & 0 \\
0\ar[r]&\mathcal{E}_{d,q}\ar[r]          &V_{\mathbb{P}^1_q}\ar[r]\ar[d]   & \mathcal{F}_{d,q}\ar[r]\ar[d] & 0  \\
       &                                       &F\ar@{=}[r]\ar[d]& F\ar[d]             & \\
       &                                       &0                                & 0                                   & \\
}
\]
The middle column represents a point $x\in Q_r$ since $\deg F=r$, and the first row corresponds
a point $y\in Q_{d,r}$, which is on the fiber $\theta^{-1}(x)$ over $x$. One checks that $\varphi(y)=q$. Thus
$\varphi$ is surjective.
\end{proof}

We put $\mathring{Q}_{d,r}:=Q_{d,r}\times_{Q_r}\mathring{Q}_r=\theta^{-1}(\mathring{Q}_r)\subset Q_{d,r}$,
and write $\phi_0:\mathring{Q}_{d,r}\to Q_d$ for the restriction of $\phi$ to $\mathring{Q}_{d,r}$.

\begin{prop}
The map $\phi_0$ factors through the inclusion $\mathring{Z}_{d,r}\hookrightarrow Q_d$.
\end{prop}
\begin{proof}
By the universal property of $\mathring{Z}_{d,r}$ (Theorem \ref{thmZ}), we only need to show that
$\bar\phi_0^*\mathcal{F}_d^\e$ is flat over $\mathring{Q}_{d,r}$ with relative degree $d-r$. Since
we have
\[
\bar\phi_0^*\mathcal{F}_d^\e=(\bar
\phi^*\mathcal{F}_d)^\e_{\mathring{Q}_{d,r}}\simeq (\mathcal{F}_{d,r})_{\mathring{Q}_{d,r}}^\e
\]
we only need to show that $(\mathcal{F}_{d,r})_{\mathring{Q}_{d,r}}^\e$ is flat over
$\mathring{Q}_{d,r}$ with relative degree $d-r$.

Note that we have a short exact sequence from the last column of the diagram (\ref{bigDiagram}):
\[
0\to \mathcal{T}_{d,r}\to \mathcal{F}_{d,r}\to \bar\theta^*\mathcal{F}_r \to 0
\]
Restricting it to $\mathbb{P}^1_{\mathring{Q}_{d,r}}$ yields an exact sequence
\[
0\to (\mathcal{T}_{d,r})_{\mathring{Q}_{d,r}}\to (\mathcal{F}_{d,r})_{\mathring{Q}_{d,r}}\to
\bar\theta^*(\mathcal{F}_r)_{\mathring{Q}_r} \to 0
\]
Then dualizing the sequence, we obtain a an exact sequence
\[
\cdots \to \bar\theta^*(\mathcal{F}_r)_{\mathring{Q}_r}^\e \to
(\mathcal{F}_{d,r})_{\mathring{Q}_{d,r}}^\e \to (\mathcal{T}_{d,r})_{\mathring{Q}_{d,r}}^\e \to 0
\]
Since $(\mathcal{F}_r)_{\mathring{Q}_r}$ is locally free, we have
$(\mathcal{F}_r)_{\mathring{Q}_r}^\e=0$, and hence
\begin{equation}\label{e4}
(\mathcal{F}_{d,r})_{\mathring{Q}_{d,r}}^\e \simeq (\mathcal{T}_{d,r})_{\mathring{Q}_{d,r}}^\e.
\end{equation}
By Proposition \ref{FlatTorsion}, $(\mathcal{T}_{d,r})^\e_{\mathring{Q}_{d,r}}$ is flat over $\mathring{Q}_{d,r}$ with relative
degree $d-r$. It follows that $(\mathcal{F}_{d,r})_{\mathring{Q}_{d,r}}^\e$ is flat over
$\mathring{Q}_{d,r}$ with relative degree $d-r$.
\end{proof}

We denote by $\mathring{\varphi}: \mathring{Q}_{d,r}\to
\mathring{Z}_{d,r}$ the map factored out from $\phi_0: \mathring{Q}_{d,r}\to Q_d$. The composition of $\mathring{\varphi}$ with the inclusion $\mathring{Z}_{d,r}\embed Z_{d,r}$ is the restriction of the morphism $\varphi$ to $\mathring{Q}_{d,r}$.

\begin{prop}
The morphism $\mathring{\varphi}: \mathring{Q}_{d,r}\to \mathring{Z}_{d,r}$ is an isomorphism.
\end{prop}
\begin{proof}
We construct a morphism $\psi: \mathring{Z}_{d,r} \to \mathring{Q}_{d,r}$, and show that $\mathring{\varphi} \psi=\id_{\mathring{Z}_{d,r}}$ and $\psi \mathring{\varphi}=\id_{\mathring{Q}_{d,r}}$.
Let $i: \mathring{Z}_{d,r}\embed Q_d$ be the inclusion. Pullback the universal exact sequence (\ref{UES-Qd}) to $\mathbb{P}^1_{\mathring{Z}_{d,r}}$:
\[
0\to \bar i^*\mathcal{E}_d\to V_{\mathbb{P}^1_{\mathring{Z}_{d,r}}}\to
\bar i^*\mathcal{F}_d\to 0
\]
Taking dual, one obtains a long exact sequence
\[
0\to (\bar i^*\mathcal{F}_d)^\vee\to V^\vee_{\mathbb{P}^1_{\mathring{Z}_{d,r}}}
\to \bar i^*\mathcal{E}_d^\vee\to \bar i^*\mathcal{F}_d^\e\to 0
\]
Let $\mathcal{G}=\ker(\bar i^*\mathcal{E}_d^\vee\to
\bar i^*\mathcal{F}_d^\e)$. Then we can break up the above sequence into two short exact sequences:
\begin{equation}\label{e5}
0\to(\bar i^*\mathcal{F}_d)^\vee\to V^\vee_{\mathbb{P}^1_{\mathring{Z}_{d,r}}}
\to \mathcal{G}\to 0, \quad 0\to\mathcal{G}\to \bar i^*\mathcal{E}_d^\vee\to
\bar i^*\mathcal{F}_d^\e \to 0
\end{equation}
By Theorem \ref{thmZ}, $\bar i^*\mathcal{F}_d^\e$ is flat over $\mathring{Z}_{d,r}$
with relative degree $d-r$. It follows that $\mathcal{G}$ and
$(\bar i^*\mathcal{F}_d)^\vee$ are both flat over $\mathring{Z}_{d,r}$, and hence they are
both locally free. Since $\bar i^*\mathcal{F}_d^\e$ has rank 0 and relative degree
$d-r$, we know that $\mathcal{G}$ has rank $k$ and relative degree $r$, and that
$(\bar i^*\mathcal{F}_d)^\vee$ has rank $n-k$ and relative degree $-r$. Taking dual
again to both sequences in (\ref{e5}), we obtain
\begin{equation}\label{e1}
0\to \mathcal{G}^\vee \to V_{\mathbb{P}^1_{\mathring{Z}_{d,r}}}\to
(\bar i^*\mathcal{F}_d)^{\vee\vee} \to 0
\end{equation}
\begin{equation}\label{e2}
0\to \bar i^*\mathcal{E}_d \to \mathcal{G}^\vee \to
(\bar i^*\mathcal{F}_d^\e)^\e \to 0
\end{equation}
We have $(\bar i^*\mathcal{F}_d)^{\vee\vee}$ is locally free of rank $n-k$ and
relative degree $r$. By the universal property of $Q_r$, the exact sequence (\ref{e1}) gives rise
to a morphism
\[
\psi_0: \mathring{Z}_{d,r}\to Q_r.
\]
such that the following diagram commutes:
\begin{equation}\label{DiagramQr}
\CD{ %
0\ar[r] & \mathcal{G}^\vee\ar[r]\ar@{=}[d] & V_{\mathbb{P}^1_{\mathring{Z}_{d,r}}}\ar[r]\ar@{=}[d] & (\bar i^*\mathcal{F}_d)^{\vee\vee}\ar[r]\ar[d]^{\simeq}& 0 \\
0\ar[r] & \bar\psi_0^*\mathcal{E}_r \ar[r] & \bar\psi_0^*V_{\mathbb{P}^1_{Q_r}}\ar[r] & \bar\psi_0^*\mathcal{F}_r\ar[r] & 0  \\
} %
\end{equation}
Here we view $\mathcal{G}^\vee$ as a subsheaf of $V_{\mathbb{P}^1_{\mathring{Z}_{d,r}}}$.
We see that $\bar\psi_0^*\mathcal{F}_r$ is locally free, and hence by Theorem \ref{thmZ}, $\psi_0$
factors through the inclusion $\mathring{Q}_{r}\hookrightarrow Q_r$. We write
\[
\psi_1: \mathring{Z}_{d,r}\to \mathring{Q}_{r}
\]
for the morphism factored out from $\psi_0$.

Taking into account of the diagram (\ref{DiagramQr}), the exact sequence (\ref{e2}) becomes
\begin{equation}\label{e3}
0\to \bar i^*\mathcal{E}_d \to \bar\psi_0^*\mathcal{E}_r \to
(\bar i^*\mathcal{F}_d^\e)^\e \to 0
\end{equation}
By Proposition \ref{FlatTorsion}, $(\bar i^*\mathcal{F}_d^\e)^\e$ is flat over
$\mathring{Z}_{d,r}$ with rank 0 and relative degree $d-r$. Thus by the universal property of
$Q_{d,r}$, the exact sequence (\ref{e3}) above induces a morphism of $Q_r$-schemes
\[
\psi_2: \mathring{Z}_{d,r}\to Q_{d,r}
\]
such that the following diagram commutes:
\begin{equation}\label{DiagramZdr}
\CD{ %
0\ar[r] & \bar i^*\mathcal{E}_d\ar[r]\ar@{=}[d]& \bar\psi_0^*\mathcal{E}_r
\ar[r]\ar@{=}[d] & (\bar i^*\mathcal{F}_d^\e)^\e \ar[r]\ar[d]^{\simeq} & 0 \\
0\ar[r] & \bar\psi_2^*\mathcal{E}_{d,r} \ar[r]& \bar\psi_2^*\bar\theta^*\mathcal{E}_r \ar[r] &
\bar\psi_2^*\mathcal{T}_{d,r}\ar[r] & 0
} %
\end{equation}
The maps $\psi_1$ and $\psi_2$ fit into the following commutative diagram
\[
\CD{ %
\mathring{Z}_{d,r}\ar@/^/[drr]^{\psi_2}\ar@/_/[ddr]_{\psi_1}\ar@{.>}[dr]|-{\psi} & & \\
 & \mathring{Q}_{d,r}\ar[d]\ar@{^(->}[r]_j & Q_{d,r}\ar[d]^{\theta} \\
 & \mathring{Q}_r\ar@{^(->}[r] & Q_r
} %
\]
Thus we obtain another $Q_r$-morphism
\[
\psi:=\psi_2\times\psi_1: \mathring{Z}_{d,r}\to Q_{d,r}\times_{Q_r}
\mathring{Q}_r=\mathring{Q}_{d,r}
\]
We now show that (i) $\mathring{\varphi} \psi=\id_{\mathring{Z}_{d,r}}$ and (ii) $\psi
\mathring{\varphi}=\id_{\mathring{Q}_{d,r}}$.

(i) To show that $\mathring{\varphi} \psi=\id_{\mathring{Z}_{d,r}}$, it suffices to show that $i\mathring{\varphi}
\psi=i$, or $\phi_0 \psi=i$, or $\phi j\psi=i$, or $\phi\psi_2=i$.
Note that we have the following commutative diagram
\[
\CD{ %
0\ar[r] & \bar\psi_2^*\mathcal{E}_{d,r}\ar@{=}[d]\ar[r]& \bar\psi_2^*V_{\mathbb{P}^1_{Q_{d,r}}}
\ar@{=}[d]\ar[r] & \bar\psi_2^*\mathcal{F}_{d,r}\ar@{..>}[d]^{\simeq}\ar[r] & 0 \\
0\ar[r] & \bar i^*\mathcal{E}_d\ar[r] &
V_{\mathbb{P}^1_{\mathring{Z}_{d,r}}}\ar[r] & \bar i^*\mathcal{F}_d \ar[r] & 0
} %
\]
where the first row is obtained by applying $\bar\psi_2^*$ to the middle row of the diagram
(\ref{bigDiagram}), and the equality
$\bar\psi_2^*\mathcal{E}_{d,r}=\bar i^*\mathcal{E}_d$ is from the diagram
(\ref{DiagramZdr}).
Thus we have a commutative diagram of equivalent quotients:
\[
\CD{ %
{\overline{\phi\psi_2}\,}^*V_{\mathbb{P}^1_{Q_d}}\ar@{=}[r]\ar@{->>}[d] &
\bar\psi_2^*\bar\phi^*V_{\mathbb{P}^1_{Q_d}} \ar@{->>}[d]\ar@{=}[r] & \bar\psi_2^*
V_{\mathbb{P}^1_{Q_{d,r}}} \ar@{=}[r]\ar@{->>}[d] & V_{\mathbb{P}^1_{\mathring{Z}_{d,r}}}
\ar@{=}[r]\ar@{->>}[d] & \bar
i^*V_{\mathbb{P}^1_{Q_d}} \ar@{->>}[d] \\
{\overline{\phi\psi_2}\,}^*\mathcal{F}_d \ar@{=}[r] & \bar\psi_2^*\bar\phi^*\mathcal{F}_d
\ar[r]^{\simeq} & \bar\psi_2^*\mathcal{F}_{d,r} \ar[r]^{\simeq} & \bar i^*\mathcal{F}_d
\ar@{=}[r] & \bar i^*\mathcal{F}_d
} %
\]
By the universal property of $Q_d$, we must have $\phi\psi_2=i$, and hence $\mathring{\varphi}
\psi=\id_{\mathring{Z}_{d,r}}$.

(ii) Let $j: \mathring{Q}_{d,r}\to Q_{d,r}$ denote the inclusion. To show that $\psi \mathring{\varphi}=\id_{\mathring{Q}_{d,r}}$, it suffices to show that $j \psi
\mathring{\varphi}=j$, or $\psi_2 \mathring{\varphi}=j$.

We first show that $\psi_2\mathring{\varphi}: \mathring{Q}_{d,r}\to Q_{d,r}$ is a $Q_r$-morphism, where $\mathring{Q}_{d,r}$ is viewed as a $Q_r$-scheme through $\theta j: \mathring{Q}_{d,r}\to Q_r$. Equivalently, we show that $\theta j=\theta \psi_2\mathring{\varphi}$, or $\theta j=\psi_0\mathring{\varphi}$. We have two commutative diagrams of the same shape with exact rows and columns:
\[
\CD{
       &                                     & 0\ar[d]                                        & 0\ar[d]                             & \\
0\ar[r]&\bar j^*\mathcal{E}_{d,r}\ar[r]\ar@{=}[d]&\bar j^*\bar\theta^*\mathcal{E}_{r}\ar[r]\ar[d]    & \bar j^*\mathcal{T}_{d,r}\ar[r]\ar@{..>}[d]     & 0 \\
0\ar[r]&\bar j^*\mathcal{E}_{d,r}\ar[r]          &V_{\mathbb{P}^1_{\mathring{Q}_{d,r}}}\ar[r]\ar[d]    & \bar j^*\bar\phi^*\mathcal{F}_d\ar[r]\ar@{..>}[d]     & 0  \\
       &                                     &\bar j^*\bar\theta^*\mathcal{F}_{r}\ar@{=}[r]\ar[d]& \bar j^*\bar\theta^*\mathcal{F}_r\ar[d]& \\
       &                                     &0                                               & 0                                   & \\
}
\qquad
\CD{
       &                                     & 0\ar[d]                                        & 0\ar[d]                             & \\
0\ar[r]&\bar{\mathring{\varphi}}^*\bar i^*\mathcal{E}_{d}\ar[r]\ar@{=}[d]&\bar{\mathring{\varphi}}^*\bar\psi_0^*\mathcal{E}_{r}\ar[r]\ar[d]    & \bar{\mathring{\varphi}}^*(\bar i^*\mathcal{F}_{d}^\e)^\e\ar[r]\ar@{..>}[d]     & 0 \\
0\ar[r]&\bar{\mathring{\varphi}}^*\bar i^*\mathcal{E}_{d}\ar[r]          &V_{\mathbb{P}^1_{\mathring{Q}_{d,r}}}\ar[r]\ar[d]    & \bar{\mathring{\varphi}}^*\bar i^*\mathcal{F}_d\ar[r]\ar@{..>}[d]     & 0  \\
       &                                     &\bar{\mathring{\varphi}}^*\bar\psi_0^*\mathcal{F}_{r}\ar@{=}[r]\ar[d]& \bar{\mathring{\varphi}}^*\bar\psi_0^*\mathcal{F}_r\ar[d]& \\
       &                                     &0                                               & 0                                   & \\
}
\]
The left diagram is obtained from (\ref{bigDiagram}) by replacing $\mathcal{F}_{d,r}$ with $\bar\phi^*\mathcal{F}_d$ using the isomorphism $\mathcal{F}_{d,r}\simeq \bar\phi^*\mathcal{F}_d$ first and applying $\bar j^*$ second, while the right diagram is obtained by combining the two exact sequences (\ref{e1}) and (\ref{e2}) first, replacing $\mathcal{G}^\vee$ with $\bar\psi_0^*\mathcal{E}_r$ and $(\bar i^*\mathcal{F}_d)^{\vee\vee}$ with $\bar\psi_0^*\mathcal{F}_r$ second, and applying $\bar{\mathring{\varphi}}^*$ third.
Note that $i\mathring{\varphi}=\phi_0=\phi j$. Hence the two middle rows are exactly the same exact sequences because they are both the pullback of the universal exact sequence of $Q_d$ by the same map $\phi_0$. Note also that $\bar j^*\mathcal{T}_{d,r}$ and $\bar{\mathring{\varphi}}^*(\bar i^*\mathcal{F}_{d}^\e)^\e$ are both torsion submodules of the same module  $\bar j^*\bar\phi^*\mathcal{F}_d=\bar{\mathring{\varphi}}^*\bar i^*\mathcal{F}_d$ since $\bar j^*\bar\theta^*\mathcal{F}_r$ and $\bar{\mathring{\varphi}}^*\bar\psi_0^*\mathcal{F}_r$ are both locally free. Therefore, $\bar j^*\mathcal{T}_{d,r}=\bar{\mathring{\varphi}}^*(\bar i^*\mathcal{F}_{d}^\e)^\e$ as submodules, and hence we have a commutative diagram
\[
\CD{
  V_{\mathbb{P}^1_{\mathring{Q}_{d,r}}}\ar@{->>}[r]\ar@{=}[d] & \bar j^*\bar\phi^*\mathcal{F}_d \ar@{->>}[r]\ar@{=}[d] & \bar j^*\bar\theta^*\mathcal{F}_r \ar[d]^{\simeq} \\
  V_{\mathbb{P}^1_{\mathring{Q}_{d,r}}}\ar@{->>}[r] & \bar{\mathring{\varphi}}^*\bar i^*\mathcal{F}_d \ar@{->>}[r] & \bar{\mathring{\varphi}}^*\bar\psi_0^*\mathcal{F}_r
}
\]
So $\theta j=\psi_0\mathring{\varphi}$ by the universal property of $Q_r$.

Next, we have a commutative diagram of equivalent quotients on $\mathbb{P}^1_{Q_{d,r}}$:
{\small
\[
\CD{ %
{\overline{\psi_2 \mathring{\varphi}}\,}^*\bar\theta^*\mathcal{E}_r \ar@{->>}[d]\ar@{=}[r] &
\bar{\mathring{\varphi}}^*\bar\psi_2^*\bar\theta^*\mathcal{E}_r \ar@{->>}[d]\ar@{=}[r] &
\bar{\mathring{\varphi}}^*\bar\psi_0^*\mathcal{E}_r \ar@{->>}[d]\ar@{=}[r] & \bar j^*\bar\theta^*\mathcal{E}_r
\ar@{->>}[d] \\
{\overline{\psi_2 \mathring{\varphi}}\,}^*\mathcal{T}_{d,r} \ar@{=}[r] &
\bar{\mathring{\varphi}}^*\bar\psi_2^*\mathcal{T}_{d,r} \ar[r]^{\simeq}  &
\bar{\mathring{\varphi}}^*(\bar i^*\mathcal{F}_d^\e)^\e \ar@{=}[r] & \bar j^*\mathcal{T}_{d,r}
}%
\]
}
By the universal property of $Q_{d,r}$, $\psi_2\mathring{\varphi}=j$ as $Q_r$-morphisms, hence
$\psi\mathring{\varphi}=\id_{\mathring{Q}_{d,r}}$.
\end{proof}

This shows that $\varphi: Q_{d,r}\to Z_{d,r}$ is a birational morphism. It then follows that
\begin{cor}
For each $r$, $Z_{d,r}$ is irreducible and of codimension $(n-k)(d-r)$, and $\mathring{Z}_{d,r}$ is nonsingular. In particular, $Z_{d,0}$ is an irreducible nonsingular subvariety of $Q_d$.
\end{cor}

\section{Successive {blowups} and main theorems}

We now perform a sequence of {blowups} on $Q_d$. Set $Q_d^{-1}:=Q_d$, $Z_{d,r}^{-1}:=Z_{d,r}$. For $r=0,\cdots,d-1$, let $Q_d^r$ be the blowup of $Q_d^{r-1}$ along $Z_{d,r}^{r-1}$ ($Q_d^r:=\Bl_{Z_{d,r}^{r-1}}Q_d^{r-1}$), $Z_{d,r}^r$ the exceptional divisor, and $Z_{d,l}^r$ the proper transform of $Z_{d,l}^{r-1}$ in $Q_d^r$ for $l\neq r$. We write $I_{Z_{d,r}^l}$ for the ideal sheaf of $Z_{d,r}^l$ in $Q_d^l$, for all $r,l$. Recall that $I_{Z_{d,r}}=I_{d,r,m_d}$.

This way of constructing of a compactification is similar to that in the construction of the space of complete quadrics and the space of complete collineations by Vainsencher. In fact, our proof will rely on a result from his construction of the space of complete collineations. To state that result, we begin with a brief review of the space of collineations. Some original notations are modified.

Let $E$ and $F$ be  two locally free sheaves on a scheme $X$. The space of complete collineations $\BBS(E,F)$ from $E$ to $F$ is defined as the projective bundle
\[
\BBS(E,F):=\mathbb{P}(\HHom(E,F))
\]
over $X$. It is equipped with a nowhere-vanishing universal homomorphism $u: E\to F\otimes \mathcal{O}_{\BBS(E,F)}(1)$. For $r=1,\cdots,N:=\min\{\rank E,\rank F\}-1$, let $D_r(E,F)$ be the schematic zero locus of $\bigwedge^{r+1} u: \bigwedge^{r+1}E\to \bigwedge^{r+1} F\otimes \mathcal{O}_{\BBS(E,F)}(r+1)$. The ideal sheaf $I_{D_r(E,F)}$ of $D_r(E,F)$ is the image of the homomorphism
\[
\HHom\bigg(\bigwedge^{r+1}E,\bigwedge^{r+1}F\bigg)^\vee\otimes\mathcal{O}_{\BBS(E,F)}(-r-1)\surj I_{D_r(E,F)}\subset \mathcal{O}_{\BBS(E,F)}
\]
induced by $\bigwedge^{r+1} u$. Let $\BBS_r(E,F):=\BBS(\bigwedge^{r+1}E,\bigwedge^{r+1}F)$. The section $\bigwedge^{r+1} u$ induces an $X$-morphism: $\BBS(E,F)\setminus D_r(E,F)\to \BBS_r(E,F)$.

Starting with $\BBS^0(E,F):=\BBS(E,F)$, $D_r^0(E,F):=D_r(E,F)$, let $\BBS^r(E,F)$ be the blowup of $\BBS^{r-1}(E,F)$ along $D_r^{r-1}(E,F)$, and denote by $D_r^r(E,F)$ the exceptional divisor, $D_l^r(E,F)$ the proper transform of $D_l^{r-1}(E,F)$ in $\BBS^r(E,F)$ for $l\neq r$. The ideal sheaf of $D_l^r(E,F)$ in $\BBS^r(E,F)$ is denoted by $I_{D_l^r(E,F)}$.

The useful result is Theorem 2.4 (8) in \cite{Vain84}, which gives a relationship between the ideal sheaf of the total transform and that of the proper transform of $D_l^{r-1}(E,F)$ in the {blowup} $\BBS^r(E,F)\to \BBS^{r-1}(E,F)$. We state the result below for the convenience of reference.
\begin{thm}[Vainsencher]\label{Vainsencher}
For $l> r\geq 1$, we have
\[
I_{D_l^{r-1}(E,F)}\cdot \mathcal{O}_{\BBS^r(E,F)}=I_{D_l^r(E,F)}\cdot (I_{D_r^r(E,F)})^{l-r+1}
\]
\end{thm}

We will use this result to show that there is a similar relationship between the ideal sheaf of the total transform and that of the proper transform of $Z_{d,l}^{r-1}$ in the {blowup} $Q_d^r\to Q_d^{r-1}$.
\begin{prop} \label{IdealFactor1}
For $l>r\geq 0$, $I_{Z_{d,l}^{r-1}}\cdot
\mathcal{O}_{Q_d^r}=I_{Z_{d,l}^r}\cdot(I_{Z_{d,r}^r})^{l-r+1}$
\end{prop}
\begin{proof}
Fix an integer $m\geq m_d$ and consider the space of collineations $\BBS(\pi_*V_{\mathbb{P}^1_{Q_d}}^\vee(m),\pi_*\mathcal{E}_d^\vee(m))$. For simplicity of notations, we put $\BBS:=\BBS(\pi_*V_{\mathbb{P}^1_{Q_d}}^\vee(m),\pi_*\mathcal{E}_d^\vee(m))$, $D_r:=D_r(\pi_*V_{\mathbb{P}^1_{Q_d}}^\vee(m),\pi_*\mathcal{E}_d^\vee(m))$, and $
D_l^r:=D_l^r(\pi_*V_{\mathbb{P}^1_{Q_d}}^\vee(m),\pi_*\mathcal{E}_d^\vee(m))$.

The nowhere vanishing section $\rho_{d,m}$ of $\HHom(\pi_*V_{\mathbb{P}^1_{Q_d}}^\vee(m),\pi_*\mathcal{E}_d^\vee(m)) $ on $Q_d$ induces a closed embedding (a morphism over $Q_d$) $[\rho_{d,m}]: Q_d\hookrightarrow \BBS$ such that $\rho_{d,m}=[\rho_{d,m}]^*u$, where $u$ is the universal homomorphism on $\BBS$:
\[
u: \pi_*V_{\mathbb{P}^1_{Q_d}}^\vee(m)\to \pi_*\mathcal{E}_d^\vee(m)\otimes \mathcal{O}_\BBS(1).
\]
Through this embedding, we will consider $Q_d$ as a closed subscheme of $\BBS$. By definition, $Z_{d,r}$ is the schematics zero locus of $\bigwedge^{\widehat{m}+r+1}\rho_{d,m}=[\rho_{d,m}]^*\bigwedge^{\widehat{m}+r+1}u$. It follows that
\[
Q_d\cap D_{\widehat m+r}=Z_{d,r}, \quad r=0,\dots,d-1
\]
and
\[
Q_d\cap D_l=\varnothing, \quad l=1,\dots,\widehat m-1.
\]
where the intersections are scheme-theoretic. The second equation means that the first $\widehat m-1$ {blowups} of $\BBS$ along $D_l$ for $l=1,\dots,\widehat m-1$ has no effect on $Q_d$. Therefore, we also have an embedding $Q_d\embed \BBS^{\widehat m-1}$, together with $Q_d\cap D_{\widehat{m}+r}^{\widehat{m}-1}=Z_{d,r}$, for $0\leq r\leq d-1$. In other words, $I_{D_{\widehat{m}+r}^{\widehat{m}-1}}\cdot \mathcal{O}_{Q_d}=I_{Z_{d,r}}$.

Starting from the $\widehat m$-th {blowup}, the {blowups} on $\BBS$ has an effect on $Q_d$. We have the following pullback diagram of {blowups}
\[
\CDR{ 
Q_d^r \ar@{^(->}[r]\ar[d] & \BBS^{\widehat{m}+r}  \ar[d] \\
\quad Q_d^{r-1} \ar@{^(->}[r] & \BBS^{\widehat{m}+r-1}
} 
\]
for all $0\leq r\leq d-1$. Further we have $Q_d^r\cap D_{\widehat{m}+l}^{\widehat{m}+r} =Z_{d,l}^r, \text{ for }l\geq r$.
Thus for $l>r$,
\[
\begin{split}
I_{Z_{d,i}^{r-1}}\cdot \mathcal{O}_{Q_d^r} & =I_{D_{\widehat{m}+i}^{\widehat{m}+r-1}}\cdot
\mathcal{O}_{Q_d^{r-1}}\cdot \mathcal{O}_{Q_d^r} = I_{D_{\widehat{m}+i}^{\widehat{m}+r-1}}\cdot \mathcal{O}_{\BBS_{\widehat{m}+i}^{\widehat{m}+r}}\cdot
\mathcal{O}_{Q_d^r} \\
& \stackrel{(*)}{=} I_{D_{\widehat{m}+i}^{\widehat{m}+r}}\cdot (I_{D_{\widehat{m}+i}^{\widehat{m}+r}})^{i-r+1} \cdot
\mathcal{O}_{Q_d^r} = I_{Z_{d,i}^r}\cdot (I_{Z_{d,r}^r})^{i-r+1}
\end{split}
\]
where $(*)$ is by Theorem \ref{Vainsencher}.
\end{proof}

Applying the above Proposition repeatedly, we obtain
\begin{cor}\label{IdealFactor}
For all $0\leq r\leq d-1$,
\[
I_{Z_{d,r}}\cdot \mathcal{O}_{Q_d^{r-1}}=I_{Z_{d,r}^{r-1}}\cdot (I_{Z_{d,r-1}^{r-1}})^2\cdot\cdots\cdot
(I_{Z_{d,1}^{r-1}})^{r} \cdot (I_{Z_{d,0}^{r-1}})^{r+1}
\]
\end{cor}

Take $m\geq m_d$. For each $l$, $\ds\bigwedge^{\widehat{m}+l+1}\rho_{d,m}$, a section of $\ds\HHom(\bigwedge^{\widehat{m}+l+1}\pi_*V_{\mathbb{P}^1_{Q_d}}^\vee(m), \bigwedge^{\widehat{m}+l+1}\pi_*\mathcal{E}_d^\vee(m))$, is nowhere vanishing on $Q_d\setminus Z_{d,l}$, hence it induces an embedding (over $Q_d$)
\[
[\bigwedge^{\widehat m+l+1}\rho_{d,m}]:Q_d\setminus Z_{d,l}\embed \BBS_{\widehat m+l}(\pi_*V_{\mathbb{P}^1_{Q_d}}^\vee(m), \pi_*\mathcal{E}_d^\vee(m))
\]
We also denote by $[\bigwedge^{\widehat m+r+1}\rho_{d,m}]$ any of its restrictions on $Q_d\setminus Z_{d,r}$ for $r>l$.

\begin{thm}
Let $m\geq m_d$. The blowup $Q_d^r$ is isomorphic to the closure of the image of the embedding
\[
\prod_{l=0}^{r}[\bigwedge^{\widehat m+l+1}\rho_{d,m}]: Q_d\setminus Z_{d,r}\embed \prod_{l=0}^{r}{}_{Q_d}\BBS_{\widehat m+l}(\pi_*V_{\mathbb{P}^1_{Q_d}}^\vee(m), \pi_*\mathcal{E}_d^\vee(m))
\]
\end{thm}
\begin{proof}
By construction, $Q_d^r$ is the blowup of $Q_d^{r-1}$ along the subscheme $Z_{d,r}^{r-1}$. Let $b: Q_d^{r-1}\to Q_d$ denote the composite {blowup}. Corollary \ref{IdealFactor} says that the ideal sheaf of the proper transform $Z_{d,r}^{r-1}$ of $Z_{d,r}$ and the ideal sheaf of the total transform $b^{-1}(Z_{d,r})$ of $Z_{d,r}$ only differ by an invertible ideal sheaf, the sheaf $(I_{Z_{d,r-1}^{r-1}})^2\cdot\cdots\cdot
(I_{Z_{d,1}^{r-1}})^{r} \cdot (I_{Z_{d,0}^{r-1}})^{r+1}$. Therefore, the blowup of $Q_d^{r-1}$ along $Z_{d,r}^{r-1}$ is the same as the blowup along the total transform $b^{-1}(Z_{d,r})$, i.e., $Q_d^r=\Bl_{b^{-1}(Z_{d,r})}Q_d^{r-1}$. Since $Z_{d,r}$ is the schematic zero locus of $\bigwedge^{\widehat m+r+1}\rho_{d,m}$ on $Q_d$, $b^{-1}(Z_{d,r})$ is the schematic zero locus of $\bigwedge^{\widehat m+r+1}b^*\rho_{d,m}$ on $Q_d^{r-1}$. So $\Bl_{b^{-1}(Z_{d,r})}Q_d^{r-1}$ is the closure of the image of the embedding induced by $\bigwedge^{\widehat m+r+1}b^*\rho_{d,m}$:
\[
Q_d^{r-1}\setminus b^{-1}(Z_{d,r})\embed \mathbb{S}_{\widehat m+r}(b^*V_{\mathbb{P}^1_{Q_d}}^\vee,b^*\mathcal{E}_d^\vee)=Q_d^{r-1}\times_{Q_d}\mathbb{S}_{\widehat m+r}(V_{\mathbb{P}^1_{Q_d}}^\vee,\mathcal{E}_d^\vee)
\]
Thus we see that the proof can be completed by induction on $r$.
\end{proof}

Note that the closure of the image of the map
\[
[\bigwedge^{\widehat m+l+1}\rho_{d,m}]: Q_d\setminus Z_{d,r}\embed \BBS_{\widehat m+l}(\pi_*V_{\mathbb{P}^1_{Q_d}}^\vee(m), \pi_*\mathcal{E}_d^\vee(m))
\]
is exactly the blowup $\Bl_{Z_{d,l}}Q_d$. Thus we obtain an easy corollary.
\begin{cor}
The blowup $Q_d^r$ is isomorphic to the closure of $Q_d\setminus Z_{d,r}$ in the product
\[
\Bl_{Z_{d,0}}Q_d\times_{Q_d}\cdots\times_{Q_d}\Bl_{Z_{d,r}}Q_d
\]
\end{cor}

Next we claim that we have a commutative diagram
\begin{equation}\label{ZembedDiagram}
\xymatrix{
\mathring{Q}_{d,r} \ar[d]_{\mathring{\varphi}} \ar@{^(->}[r]^-{\alpha} &
\ds\prod_{l=0}^{r-1}{}_{Q_{d,r}} \BBS_{\widehat{m}+l}\big(\theta^*\pi_*V_{\mathbb{P}^1_{Q_{r}}}^\vee(m),\theta^*\pi_*\mathcal{E}_{r}^\vee(m)\big) \ar@{^(->}[d]^{\beta} \\
\mathring{Z}_{d,r} \ar@{^(->}[r]^-{\gamma} & \ds\prod_{l=0}^{r-1}{}_{Q_d}
\BBS_{\widehat{m}+l}\big(\pi_*V_{\mathbb{P}^1_{Q_d}}^\vee(m),\pi_*\mathcal{E}_d^\vee(m)\big)
}
\end{equation}
for all $m\gg 0$ ($m\geq m_d$ is not sufficient). The map $\alpha$ is induced by the nowhere vanishing sections $\bigwedge^{\widehat m+l+1}\theta^*\rho_{r,m}$, $l=0,\dots,r-1$, on $\mathring{Q}_{d,r}$. We see that it is the pullback by $\theta:Q_{d,r}\to Q_r$ of the embedding $\mathring{Q}_r\embed \prod_{l=0}^{r-1}{}_{Q_{d,r}} \BBS_{\widehat{m}+l}\big(\theta^*\pi_*V_{\mathbb{P}^1_{Q_{r}}}^\vee(m),\theta^*\pi_*\mathcal{E}_{r}^\vee(m)\big)$ induced by the sections $\bigwedge^{\widehat m+l+1}\rho_{r,m}$:
\[
\begin{split}
\mathring{Q}_{d,r}=Q_{d,r}\times_{Q_{r}}\mathring{Q}_{r}
&\hookrightarrow Q_{d,r}\times_{Q_{r}}\prod_{l=0}^{r-1}{}_{Q_{r}}
\BBS_{\widehat{m}+l}\big(\pi_*V_{\mathbb{P}^1_{
Q_{r}}}^\vee(m),\pi_*\mathcal{E}_{r}^\vee(m)\big)=\\ &\prod_{l=0}^{r-1}{}_{Q_{d,r}} \BBS_{\widehat{m}+l}\big(\theta^*\pi_*V_{\mathbb{P}^1_{Q_{r}}}^\vee(m),\theta^*\pi_*\mathcal{E}_{r}^\vee(m)\big)
\end{split}
\]
Thus $\alpha$ is an embedding, and the closure of the image of $\alpha$ is exactly
$Q_{d,r}\times_{Q_{r}}Q_{r}^{r-1}$. The map $\gamma$ is defined as the composition:
\[
\mathring{Z}_{d,r}\hookrightarrow Q_d\setminus Z_{d,r-1}\hookrightarrow
\prod_{l=0}^{r-1}{}_{Q_d}\BBS_{\widehat{m}+l}\big(\pi_*V_{\mathbb{P}^1_{Q_d}}^\vee(m),\pi_*\mathcal{E}_d^\vee(m)\big)
\]
where the second map is induced by the nowhere vanishing sections $\bigwedge^{\widehat m+l+1}\rho_{d,m}$, $l=0,\dots,r-1$.
Hence $\gamma$ is also an embedding, and the closure of the image of $\gamma$ is exactly
$Z_{d,r}^{r-1}$.

To complete the diagram (\ref{ZembedDiagram}), we also need
\begin{prop}\label{beta_exist}
There exists an embedding $\beta$ which makes the diagram (\ref{ZembedDiagram}) commute.
\end{prop}

Because it takes much space to define the map $\beta$ and its exact definition will not be used in the proofs of the following theorems, we leave the proof of Proposition \ref{beta_exist} to the next section. For the moment, we assume this proposition is true.
\begin{prop}
The embedding $\beta$ maps $Q_{d,r}\times_{Q_r}Q_r^{r-1}$ isomorphically onto $Z_{d,r}^{r-1}$.
\end{prop}
\begin{proof}
We know the closure of the image of $\alpha$ is $Q_{d,r}\times_{Q_r}Q_r^{r-1}$. Since $\beta$ is an embedding, the closure of the image of $\beta\alpha$ is isomorphic to $Q_{d,r}\times_{Q_r}Q_r^{r-1}$. On the other hand, since $\mathring{\varphi}$ is an isomorphism, the closure of the image of $\gamma\mathring{\varphi}$ is $Z_{d,r}^{r-1}$. By the commutativity of the diagram (\ref{ZembedDiagram}), we have an isomorphism $Q_{d,r}\times_{Q_r}Q_r^{r-1}\cong Z_{d,r}^{r-1}$ given by $\beta$.
\end{proof}

We are now ready to prove the first the main result.
\begin{thm}\label{MainTheorem}
For $0\leq r\leq d-1$, $Z_{d,r}^{r-1}$ and $Q_d^r$ are both nonsingular.
\end{thm}
\begin{proof}
We prove this by induction on $r$ (but for all $d\geq r+1$). When $r=0$, $Z_{d,r}^{r-1}=Z_{d,0}$ is nonsingular, hence $Q_d^0$, the blowup of $Q_d$ along $Z_{d,0}$, is also nonsingular. Assume the statement is true for $r=N-1$, that is, $Z_{d,N-1}^{N-2}$ and $Q_d^{N-1}$ are nonsingular for all $d\geq N$. We prove that the statement is also true for $r=N$. By the above theorem, $Z_{d,N}^{N-1}\cong Q_{d,N}\times_{Q_{N}}Q_N^{N-1}$, and by induction hypothesis, $Q_N^{N-1}$ is nonsingular, hence $Z_{d,N}^{N-1}$ is nonsingular. $Q_d^N$ is nonsingular because it is the blowup of the nonsingular variety $Q_d^{N-1}$ along the nonsingular subvariety $Z_{d,N}^{N-1}$. This completes the proof.
\end{proof}

\begin{prop}\label{isobybeta}
The isomorphism $\beta: Q_{d,r}\times_{Q_{r}}Q_{r}^{r-1}\to Z_{d,r}^{r-1}$ induces isomorphisms
\[
\beta: Q_{d,r}\times_{Q_{r}}Z_{r,l}^{r-1}\stackrel{\simeq}{\longrightarrow} Z_{d,r}^{r-1}\cap Z_{d,l}^{r-1},\quad\text{ for } l=0,\cdots,r-1.
\]
\end{prop}

\begin{proof}
For simplicity, we put $Q_{d,r}^{r-1}:=Q_{d,r}\times_{Q_{r}} Q_{r}^{r-1}$. Consider the following
commutative diagram
\[
\CDR{
  Q_{r}^{r-1}\ar[d] & Q_{d,r}^{r-1}\ar[l]\ar[d]\ar[r] & Q_d^{r-1}\ar[d] \\
  Q_{r} & Q_{d,r}\ar[l]_\theta\ar[r]^{\phi} & Q_d
}
\]
The ideal sheaf of $Q_{d,r}\times_{Q_{r}}Z_{r,l}^{r-1}$ in $Q_{d,r}^{r-1}$ is $I_{Z_{r,l}^{r-1}}\cdot\mathcal{O}_{Q_{d,r}^{r-1}}$. Because of the isomorphism $\beta:Q_{d,r}\times_{Q_{r}}Q_{r}^{r-1}\to Z_{d,r}^{r-1}$, we see that the ideal sheaf of the preimage $\beta^{-1}(Z_{d,r}^{r-1}\cap Z_{d,l}^{r-1})$ in $Q_{d,r}^{r-1}$ is $I_{Z_{d,l}^{r-1}}\cdot \mathcal{O}_{Q_{d,r}^{r-1}}$. Thus to show that $\beta$ induces isomorphisms $Q_{d,r}\times_{Q_{r}}Z_{r,l}^{r-1}{\simeq} Z_{d,r}^{r-1}\cap Z_{d,l}^{r-1}$, it suffices to show that
\begin{equation}\label{IdealEquality2}
I_{Z_{d,l}^{r-1}}\cdot \mathcal{O}_{Q_{d,r}^{r-1}}=I_{Z_{r,l}^{r-1}}\cdot \mathcal{O}_{Q_{d,r}^{r-1}}
\end{equation}

We now prove it by induction on $l$. When $l=0$, we have
\[
\begin{split}
\ I_{Z_{d,0}^{r-1}}\cdot \mathcal{O}_{Q_{d,r}^{r-1}} 
& =(I_{Z_{d,0}}\cdot \mathcal{O}_{Q_{d,r}})\cdot \mathcal{O}_{Q_{d,r}^{r-1}}
\stackrel{(*)}{=}(I_{Z_{r,0}}\cdot \mathcal{O}_{Q_{d,r}})\cdot \mathcal{O}_{Q_{d,r}^{r-1}} \\
& =(I_{Z_{r+1,0}}\cdot \mathcal{O}_{Q_{r}^{r-1}})\cdot \mathcal{O}_{Q_{d,r}^{r-1}}=I_{Z_{r,0}^{r-1}}\cdot \mathcal{O}_{Q_{d,r}^{r-1}}
\end{split}
\]
where $(*)$ is by Proposition \ref{IdealEquality}.

Assume (\ref{IdealEquality2}) holds for all $l\leq N-1$ for some $N\leq r-1$, and we prove it also holds for $l=N$.
By Corollary \ref{IdealFactor}, we have
\[
\begin{split}
I_{Z_{d,N}}\cdot \mathcal{O}_{Q_d^{N-1}}& =I_{Z_{d,N}^{N-1}}\cdot(I_{Z_{d,N-1}^{N-1}})^2\cdot
\dots\cdot(I_{Z_{d,0}^{N-1}})^{N+1} \\
I_{Z_{r,N}}\cdot \mathcal{O}_{Q_{r}^{N-1}}&=I_{Z_{r,N}^{N-1}}\cdot(I_{Z_{r,N-1}^{N-1}})^2\cdot
\dots\cdot(I_{Z_{r,0}^{N-1}})^{N+1} \\
\end{split}
\]
Hence we have
\begin{equation}\label{e6}
\begin{split}
&\quad\ I_{Z_{d,N}^{r-1}}\cdot(I_{Z_{d,N-1}^{r-1}})^2\cdot \dots\cdot(I_{Z_{d,0}^{r-1}})^{N+1}\cdot \mathcal{O}_{Q_{d,r}^{r-1}} \\
&= \big(I_{Z_{d,N}^{N-1}}\cdot(I_{Z_{d,N-1}^{N-1}})^2\cdot \dots\cdot(I_{Z_{d,0}^{N-1}})^{N+1}\cdot
\mathcal{O}_{Q_d^{r-1}}\big)\cdot\mathcal{O}_{Q_{d,r}^{r-1}} \\
&= ((I_{Z_{d,N}}\cdot \mathcal{O}_{Q_{d}^{N-1}})\cdot \mathcal{O}_{Q_d^{r-1}}) \cdot \mathcal{O}_{Q_{d,r}^{r-1}} = (I_{Z_{d,N}}\cdot \mathcal{O}_{Q_d^{r-1}}) \cdot \mathcal{O}_{Q_{d,r}^{r-1}} \\
&= (I_{Z_{d,N}}\cdot \mathcal{O}_{Q_{d,r}})\cdot \mathcal{O}_{Q_{d,r}^{r-1}} = (I_{Z_{r,N}} \cdot \mathcal{O}_{Q_{d,r}})\cdot \mathcal{O}_{Q_{d,r}^{r-1}} \\
&= ((I_{r,N} \cdot \mathcal{O}_{Q_{r}^{N-1}})\cdot \mathcal{O}_{Q_{r}^{r-1}}) \cdot \mathcal{O}_{Q_{d,r}^{r-1}} \\
&= (I_{Z_{r,N}^{N-1}}\cdot(I_{Z_{r,N-1}^{N-1}})^2\cdot \dots\cdot(I_{Z_{r,0}^{N-1}})^{N+1} \cdot \mathcal{O}_{Q_{r}^{r-1}}) \cdot
\mathcal{O}_{Q_{d,r}^{r-1}} \\
&= I_{Z_{r,N}^{r-1}}\cdot(I_{Z_{r,N-1}^{r-1}})^2\cdot \dots\cdot(I_{Z_{r,0}^{r-1}})^{N+1} \cdot \mathcal{O}_{Q_{d,r}^{r-1}}
\end{split}
\end{equation}
By induction hypothesis, we have an equality of invertible ideal sheaves
\[
(I_{Z_{d,N-1}^{r-1}})^2\cdot \dots\cdot(I_{Z_{d,0}^{r-1}})^{N+1}\cdot \mathcal{O}_{Q_{d,r}^{r-1}}=
(I_{Z_{r,N-1}^{r-1}})^2\cdot \dots\cdot(I_{Z_{r,0}^{r-1}})^{N+1} \cdot \mathcal{O}_{Q_{d,r}^{r-1}}
\]
We can eliminate this invertible ideal sheaf from both sides of the equality (\ref{e6}) and
obtain
\[
 I_{Z_{d,N}^{r-1}}\cdot \mathcal{O}_{Q_{d,r}^{r-1}}=I_{Z_{r,N}^{r-1}}\cdot\mathcal{O}_{Q_{d,r}^{r-1}}.
\]
Thus the statement is true for $N$.
\end{proof}

An easy corollary is as follows:
\begin{cor}\label{IntersectionIsomorphism}
For any $N$ distinct integers $l_1,\dots,l_N\in\{0,\cdots,r-1\}$, $\beta$ maps the
scheme-theoretic intersection $\ds Z_{d,r}^{r-1}\cap\bigcap_{j=1}^N Z_{d,l_j}^{r-1}$ isomorphically onto
$Q_{d,r}\times_{Q_{r}} \ds\bigcap_{j=1}^N Z_{r,l_j}^{r-1}.$
\end{cor}

\begin{thm}\label{SmoothDivisor}
For $0\leq l\leq d-1$ and $l\leq r\leq d-1$, $Z_{d,l}^r$ is nonsingular, irreducible, and of codimension one in $Q_d^r$.
\end{thm}
\begin{proof}
We fix $l$ and prove that $Z_{d,l}^r$ is nonsingular, irreducible, and of codimension one for all $r\geq l$. By definition, $Z_{d,l}^l$ is the exceptional divisor in the blowup $Q_d^l\to Q_d^{l-1}$ with the nonsingular irreducible blowup center $Z_{d,l}^{l-1}$. Hence $Z_{d,l}^l$ is nonsingular, irreducible, and of codimension one. $Z_{d,l}^{l+1}$, the proper transform of $Z_{d,l}^l$ in $Q_d^{l+1}$, can be considered as the blowup of $Z_{d,l}^l$ along its nonsingular subvariety $Z_{d,l+1}^l\cap Z_{d,l}^l$, hence it is nonsingular and of codimension one. Repeat this argument, we conclude that $Z_{d,l}^{l},\cdots,Z_{d,l}^{d-1}$ are all nonsingular and of codimension one.
\end{proof}

It remains to prove the transversality.

\begin{thm}\label{transversal}
For $0\leq r\leq d-1$, $Z_{d,0}^r,\cdots, Z_{d,r}^r$ intersect transversally in $Q_d^r$.
\end{thm}
\begin{proof}
We prove it by induction on $r$ (for all $d>r$). When $r=0$, the statement is trivial. Assume that the statement is true for $r-1$. We first show that $Z_{d,r}^{r-1}, Z_{d,0}^{r-1},\dots,Z_{d,r-1}^{r-1}$ transversally. Since $ Z_{d,0}^{r-1},\dots,Z_{d,r-1}^{r-1}$ intersect transversally by induction hypothesis, we only need to show that $Z_{d,r}^{r-1}$ intersect with $Z_{d,0}^{r-1},\dots,Z_{d,r-1}^{r-1}$ transversally. Let $x\in Z_{d,r}^{r-1}$, and suppose $Z_{d,l_1}^{r-1},\dots,Z_{d,l_N}^{r-1}$ are those from the collection $\{Z_{d,l}^{r-1}\,|\,0\leq l\leq r-1\}$ which
pass through $x$. We show that $Z_{d,r}^{r-1}$ intersect with $Z_{d,l_1}^{r-1},\dots,Z_{d,l_N}^{r-1}$ transversally at $x$ by calculating the dimensions of tangent spaces. Since the tangent space of a scheme-theoretic intersection of some subschemes is equal to the intersection the the tangent spaces of those subschemes, we have
\[
T_x{Z_{d,r}^{r-1}}\cap \bigcap_{j=1}^N
T_x{Z_{d,l_j}^{r-1}}=T_x\big({Z_{d,r}^{r-1}\cap\bigcap_{j=1}^N
Z_{d,l_j}^{r-1}}\big)\simeq T_{\beta^{-1}(x)}\big({Q_{d,r}\times_{Q_r}\bigcap_{j=1}^N Z_{r,l_j}^{r-1}}\big)
\]
where the last equality is by Corollary \ref{IntersectionIsomorphism}. Let $y\in Q_r^{r-1}$ be the image of $\beta^{-1}(x)$ under the projection $Q_{d,r}\times_{Q_r}Q_r^{r-1}\to Q_r^{r-1}$. By induction hypothesis, $Z_{r,l}^{r-1}$, $l=0,\dots,r-1$, intersect transversally in $Q_r^{r-1}$. Thus we have
\[
\codim \big(\bigcap_{j=1}^N T_y{Z_{r,l_j}^{r-1}}\big)=\sum_{j=1}^N\codim T_y{Z_{d,l_j}^{r-1}}=N
\]
or
\[
\dim\big(\bigcap_{j=1}^N T_y{Z_{r,l_j}^{r-1}}\big)=\dim Q_r^{r-1}-N=\dim Q_r-N
\]
Therefore
\[
\begin{split}
&\quad\dim\big(T_x{Z_{d,r}^{r-1}}\cap \bigcap_{j=1}^N T_x{Z_{d,l_j}^{r-1}}\big)=\dim
T_{\beta^{-1}(x)}\big({Q_{d,r}\times_{Q_r}\bigcap_{j=1}^N Z_{r,l_j}^{r-1}}\big)\\
&=(\dim Q_{d,r}-\dim Q_r)+\dim\big(\bigcap_{j=1}^N T_y{Z_{r,l_j}^{r-1}}\big)=(\dim Q_{d,r}-\dim Q_r)+(\dim Q_r-N)\\
&=\dim Q_{d,r}-N=\dim Z_{d,r}^{r-1}-N=\dim T_x Z_{d,r}^{r-1}-N
\end{split}
\]
and hence
\[
\begin{split}
&\quad \codim \big(T_x{Z_{d,r}^{r-1}}\cap \bigcap_{j=1}^N T_x{Z_{d,l_j}^{r-1}}\big)=\dim Q_d^{r-1}-(\dim T_x Z_{d,r}^{r-1}-N)\\
&=\codim T_xZ_{d,r}^{r-1}+N =\codim T_xZ_{d,r}^{r-1}+\sum_{j=1}^N\codim T_x Z_{d,l_j}^{r-1}
\end{split}
\]
It follows that $Z_{d,r}^{r-1}$ intersect with $Z_{d,0}^{r-1},\dots,Z_{d,r-1}^{r-1}$ transversally at $x$. Since $x$ is an arbitrary point, we know $Z_{d,r}^{r-1}, Z_{d,0}^{r-1},\dots,Z_{d,r-1}^{r-1}$ intersect transversally. Since transversality is preserved under {blowups} along a
nonsingular center, we obtain the statement for $r$.
\end{proof}

For $l\leq r-1$, the proper transform $Z_{d,l}^r$ of $Z_{d,l}^{r-1}$, which can be regarded as the
blowup of $Z_{d,l}^{r-1}$ along the nonsingular subvariety $Z_{d,l}^{r-1}\cap Z_{d,r}^{r-1}$, is a
nonsingular and of codimension one in $Q_d^r$. Thus Theorem \ref{Goal} follows easily from the combination of Theorem \ref{MainTheorem}, \ref{SmoothDivisor} and \ref{transversal}.

\section{The definition of the map $\beta$}

This section is a proof of Proposition \ref{beta_exist}. The first step is to construct the map $\beta$ in the diagram (\ref{ZembedDiagram}). For this, we need another relative Quot scheme $Q_{d^\vee,r}$ defined as
\[
Q_{d^\vee,r}:=\Quot^{0,d-r}_{\mathcal{E}_d^\vee/\mathbb{P}^1_{Q_d}/Q_d}
\]
with structure morphism $\vartheta: Q_{d^\vee,r}\to Q_d$. It comes equipped with a universal exact
sequence on $\mathbb{P}^1_{Q_{d^\vee,r}}$:
\[
0\to \mathcal{E}_{d^\vee,r}\to \bar\vartheta^*\mathcal{E}_d^\vee \to \mathcal{T}_{d^\vee,r}\to 0
\]
where $\mathcal{T}_{d^\vee,r}$ is flat over $Q_d$ with rank 0 and relative degree $d-r$.

Recall that the relative Quot scheme $Q_{d,r}$ has a universal exact sequence on
$\mathbb{P}^1_{Q_{d,r}}$:
\[
0\to \mathcal{E}_{d,r}\to \bar\theta^*\mathcal{E}_{r}\to \mathcal{T}_{d,r}\to 0
\]
together with a homomorphism $\varphi: Q_{d,r}\to Q_d$ such that
\[
\mathcal{E}_{d,r}=\bar\varphi^*\mathcal{E}_d.
\]
Dualizing the universal exact sequence of $Q_{d,r}$, we obtain
\[
0\to \bar\theta^*\mathcal{E}_{r}^\vee \to \bar\varphi^*\mathcal{E}_d^\vee \to
\mathcal{T}_{d,r}^\varepsilon\to 0
\]
where $\mathcal{T}_{d,r}^\varepsilon$ is flat over $Q_{d,r}$ of relative degree $d-r$ and rank 0.
By the universal property of $Q_{d^\vee,r}$, we see that the above sequence gives rise to a
$Q_d$-morphism:
\[
\eta: Q_{d,r}\to Q_{d^\vee,r}
\]
such that the following diagram commutes:
\begin{equation}\label{CD-eta}
\CD{
0\ar[r]&\bar\theta^*\mathcal{E}_{r}^\vee\ar[r]\ar@{=}[d]&\bar \varphi^*\mathcal{E}_d^\vee \ar[r]\ar@{=}[d]&\mathcal{T}_{d,r}^\varepsilon\ar[r]\ar[d]^{\simeq} & 0  \\
0\ar[r]&\bar\eta^*\mathcal{E}_{d^\vee,r}\ar[r]   &\bar\eta^*\bar \vartheta^*\mathcal{E}_d^\vee\ar[r] &\bar\eta^*\mathcal{T}_{d^\vee,r}\ar[r] &0\\
}
\end{equation}

\begin{prop}\label{eta}
The morphism $\eta$ is a closed embedding.
\end{prop}
\begin{proof}
We will show that $\eta$ induces an isomorphism from $Q_{d,r}$ to a closed subscheme of
$Q_{d^\vee,r}$. The map $\eta$ fits into the following commutative diagram:
\[
\CD{ %
Q_{d,r} \ar[d]^{\theta}\ar[r]^{\eta}\ar[dr]|-{\phi} & Q_{d^\vee,r} \ar[d]^{\vartheta} \\
Q_r & Q_d
} %
\]
We have the following exact sequence on $\mathbb{P}^1_{Q_d}$:
\[
V_{\mathbb{P}^1_{Q_d}}^\vee \to \mathcal{E}_d^\vee \to \mathcal{F}_d^\varepsilon\to 0
\]
Pull back the sequence by $\bar\vartheta$ to $\mathbb{P}^1_{Q_{d^\vee,r}}$:
\[
V_{\mathbb{P}^1_{Q_{d^\vee,r}}}^\vee \to \bar\vartheta^*\mathcal{E}_d^\vee \to \bar\vartheta^*\mathcal{F}_d^\varepsilon\to 0
\]
and let
$\mathcal{T}$ be the cokernel of the composite map
\[
V_{\mathbb{P}^1_{Q_{d^\vee,r}}}^\vee \to \bar\vartheta^*\mathcal{E}_d^\vee\to
\mathcal{T}_{d^\vee,r}
\]
As a quotient of the torsion sheaf $\mathcal{T}_{d^\vee,r}$, $\mathcal{T}$ is also torsion. We
have a commutative diagram:
\begin{equation} \label{diagram}
\CD{
 & 0 \ar[d] \\
 & \mathcal{E}_{d^\vee,r}\ar[d] \\
V_{\mathbb{P}^1_{Q_{d^\vee,r}}}^\vee \ar[r]\ar@{=}[d] & \bar\vartheta^*\mathcal{E}_d^\vee
\ar[r]\ar[d] & \bar\vartheta^*\mathcal{F}_d^\varepsilon \ar[r]\ar@{..>}[d] & 0 \\
V_{\mathbb{P}^1_{Q_{d^\vee,r}}}^\vee \ar[r] & \mathcal{T}_{d^\vee,r}\ar[r]\ar[d] &
\mathcal{T}\ar[r]\ar[d]
& 0 \\
 & 0 & 0 \\
}
\end{equation}
where the two rows and the second column are exact. The exactness of the third column is induced
from the the exactness of the second column.

For each point $q\in Q_{d^\vee,r}$, the restriction of the last row to $\mathbb{P}^1_q$ is an exact sequence
\[
V_{\mathbb{P}^1_q}^\vee \to (\mathcal{T}_{d^\vee,r})_q\to \mathcal{T}_q \to 0
\]
since pullback is a right exact functor.
Hence the fiber $\mathcal{T}_q$ on $\mathbb{P}^1_q$ is a torsion sheaf of degree at most $d-r$.
Consider the flattening stratification of $Q_{d^\vee,r}$ by $\mathcal{T}$, and let $Z$ be the stratum over which $\mathcal{T}$ has relative degree $d-r$. Then $Z$ is a closed subscheme of $Q_{d^\vee,r}$, and it satisfies the universal property:

For any morphism $f: X\to Q_{d^\vee,r}$, the pull-back $\bar f^*\mathcal{T}$ is flat over $X$ with
relative degree $d-r$ if and only if $f$ factors through the inclusion $i: Z\hookrightarrow
Q_{d^\vee,r}$.

We now show that $\eta: Q_{d,r}\to Q_{d^\vee,r}$ factors through $Z$. The pull-back of the above
commutative diagram under $\bar\eta$ fits into the following commutative diagram
\[
\CD{ %
V^\vee_{\mathbb{P}^1_{Q_{d,r}}} \ar[r]\ar@{=}[d] & \bar\theta^*\mathcal{E}_r^\vee \ar[d] \\
V^\vee_{\mathbb{P}^1_{Q_{d,r}}} \ar[r]\ar@{=}[d] & \bar\phi^*\mathcal{E}_d^\vee \ar[d] \\
V^\vee_{\mathbb{P}^1_{Q_{d,r}}} \ar[r] & \bar\eta^*\mathcal{T}_{d^\vee,r} \ar[r] & \bar\eta^*
\mathcal{T} \ar[r] & 0
} %
\]
where the second column and the third row are exact sequences. Therefore we see that
$V^\vee_{\mathbb{P}^1_{Q_{d,r}}} \to \bar\eta^*\mathcal{T}_{d^\vee,r}$ from the third row is a
zero map. So $\bar\eta^*\mathcal{T}_{d^\vee,r}=\bar\eta^* \mathcal{T}$, and hence $\bar\eta^*
\mathcal{T}$ is flat over $Q_{d,r}$ with relative degree $d-r$. By the universal property of $Z$,
the map $\eta$ factors through the inclusion $i: Z\hookrightarrow Q_{d^\vee,r}$. We denote
the map factored out from $\eta$ as
\[
\eta': Q_{d,r}\to Z.
\]
Thus, $\eta=i\eta'$.

Now pullback the last row of (\ref{diagram}) to $\mathbb{P}^1_Z$ to get an exact sequence
\[
V_{\mathbb{P}^1_Z}^\vee \to (\mathcal{T}_{d^\vee,r})_Z\to \mathcal{T}_Z \to 0
\]
By the definition of $Z$, $\mathcal{T}_Z$ is flat over $Z$ with relative degree $d-r$. Let
$\mathcal{K}$ be the kernel of the homomorphism $(\mathcal{T}_{d^\vee,r})_Z\to \mathcal{T}_Z$. Then
for every point $q\in Z$, we have a short exact sequence on $\mathbb{P}^1_q$:
\[
0\to \mathcal{K}_q\to (\mathcal{T}_{d^\vee,r})_q\to \mathcal{T}_q \to 0
\]
We see that $\mathcal{K}_q$ is a torsion sheaf of degree 0 on $\mathbb{P}^1_q$, hence
$\mathcal{K}_q=0$. It follows that $\mathcal{K}=0$, and therefore
\[
\Img(V_{\mathbb{P}^1_Z}^\vee \to (\mathcal{T}_{d^\vee,r})_Z)=\mathcal{K}=0,\quad
(\mathcal{T}_{d^\vee,r})_Z=\mathcal{T}_Z
\]
So we have the following commutative diagram on $\mathbb{P}^1_Z$:
\[
\CD{ %
 & 0\ar[d] & 0\ar[d] \\
V_{\mathbb{P}^1_Z}^\vee \ar@{..>}[r]\ar@{=}[d] & (\mathcal{E}_{d^\vee,r})_{Z}\ar@{..>}[r]\ar[d] & \mathcal{T}' \ar[r]\ar[d] & 0 \\
V_{\mathbb{P}^1_Z}^\vee \ar[r]\ar[d] & (\bar\vartheta^*\mathcal{E}_d^\vee)_{Z}\ar[r]\ar[d] &
(\bar\vartheta^*\mathcal{F}_d^\varepsilon)_{Z}\ar[r]\ar[d] & 0 \\
0 \ar[r] & (\mathcal{T}_{d^\vee,r})_{Z} \ar@{=}[r]\ar[d] & \mathcal{T}_{Z}\ar[r]\ar[d] & 0 \\
& 0 & 0
} %
\]
with all columns and the last two rows being exact, where
$\mathcal{T}':=\ker((\bar\vartheta^*\mathcal{F}_d^\varepsilon)_{Z} \to \mathcal{T}_{Z})$. Then
$\mathcal{T}'$ is torsion since $(\bar\vartheta^*\mathcal{F}_d^\varepsilon)_{Z}$ is so. The dotted
arrows in the first row are the induced maps. The first row is forced to be exact as well.

Taking dual of the first row, we get an exact sequence
\[
0\to (\mathcal{E}_{d^\vee,r})_{Z}^\vee \to V_{\mathbb{P}^1_Z}\to \mathcal{F}\to 0
\]
where $\mathcal{F}:=\Coker((\mathcal{E}_{d^\vee,r})_{Z}^\vee \to V_{\mathbb{P}^1_Z})$. We now
show that $\mathcal{F}$ is flat over $Z$ with rank $n-k$ and relative degree $r$. We only need to show that $(\mathcal{E}_{d^\vee,r})_q^\vee \to
V_{\mathbb{P}^1_q}$ is injective for every point $q\in Z$. We restrict the first row and last
column of the previous diagram to the fiber $\mathbb{P}^1_q$:
\[
\CD{ %
 & & 0\ar[d] \\
V_{\mathbb{P}^1_q}^\vee \ar[r] & (\mathcal{E}_{d^\vee,r})_q\ar[r] & \mathcal{T}'_q \ar[r]\ar[d] & 0 \\
 & & (\bar\vartheta^*\mathcal{F}_d^\varepsilon)_q\ar[d] \\
& & \mathcal{T}_q\ar[d]  \\
&  & 0
} %
\]
Since pull-back is right exact, the row is exact and the column is right exact. Note that
$\mathcal{T}_Z$ is flat over $Z$, so the column is exact, and hence, $\mathcal{T}'_q$ is torsion.
Taking dual of the row, we get an exact sequence
\[
0\to (\mathcal{E}_{d^\vee,r})_q^\vee \to V_{\mathbb{P}^1_q}
\]
So $\mathcal{F}$ is flat over $Z$ (see \cite{Huy}, Lemma 2.1.4), and we have an exact sequence
\[
0\to (\mathcal{E}_{d^\vee,r})_q^\vee \to V_{\mathbb{P}^1_q} \to \mathcal{F}_q\to 0
\]
Since $(\mathcal{E}_{d^\vee,r})_q^\vee$ has rank $k$ and degree $-r$, we have $\mathcal{F}_q$ has
rank $n-k$ and degree $r$.

By the universal property of $Q_r$, the above exact sequence induces a morphism
\[\xi: Z\to Q_r\] such
that the following diagram commutes:
\[
\CD{ %
0\ar[r]& (\mathcal{E}_{d^\vee,r})_{Z}^\vee \ar[r]\ar@{=}[d]& V_{\mathbb{P}^1_Z} \ar[r]\ar@{=}[d] & \mathcal{F} \ar[r]\ar[d]^{\simeq} & 0 \\
0\ar[r]& \bar\xi^*\mathcal{E}_r \ar[r] & \bar\xi^*V_{\mathbb{P}^1_{Q_r}} \ar[r] &
\bar\xi^*\mathcal{F}_r\ar[r] & 0
} %
\]
Restricting the universal exact sequence of $Q_{d^\vee,r}$ to $\mathbb{P}^1_Z$ and dualizing
it, we obtain an exact sequence:
\[
\CD{ %
0 \ar[r] & (\bar\vartheta^*\mathcal{E}_d)_{Z} \ar[r] &
(\mathcal{E}_{d^\vee,r})_{Z}^\vee \ar[r]\ar@{=}[d] & (\mathcal{T}_{d^\vee,r})_{Z}^\varepsilon \ar[r] & 0 \\
& & \bar\xi^*\mathcal{E}_r
} %
\]
Note that $(\mathcal{T}_{d^\vee,r})_{Z}^\varepsilon$ is flat over $Z$ with rank 0 and relative
degree $d-r$. Thus the above exact sequence induces a morphism $\zeta: Z\to Q_{d,r}$ of
$Q_r$-schemes such that
\[
\CD{ %
0 \ar[r] & (\bar\vartheta^*\mathcal{E}_d)_{Z} \ar[r]\ar@{=}[d] &
(\mathcal{E}_{d^\vee,r})_{Z}^\vee \ar[r]\ar@{=}[d] & (\mathcal{T}_{d^\vee,r})_{Z}^\varepsilon \ar[r]\ar[d]^{\simeq} & 0 \\
0 \ar[r] & \bar\zeta^*\mathcal{E}_{d,r}\ar[r] & \bar\zeta^*\bar\theta^*\mathcal{E}_r
\ar[r]\ar@{=}[d] & \bar\zeta^*\mathcal{T}_{d,r}\ar[r] & 0 \\
& & \bar\xi^*\mathcal{E}_r
} %
\]
We now show that (i) $\zeta\eta'=\id_{Q_{d,r}}$ and (ii) $\eta'\zeta=\id_Z$.

(i) We have the following commutative diagram of equivalent quotients:
\[
\CD{ %
 {\overline{\zeta\eta'}\,}^*\bar\theta^*\mathcal{E}_r \ar@{->>}[d]\ar@{=}[r] &
 \bar\eta'^*\bar \zeta^*\bar\theta^*\mathcal{E}_r \ar@{->>}[d]\ar@{=}[r] &
 \bar\eta'^*\bar \xi^*\mathcal{E}_r \ar@{->>}[d]\ar@{=}[r] &
 \bar\theta^*\mathcal{E}_r \ar@{->>}[d]\ar@{=}[r] &
 \bar\theta^*\mathcal{E}_r \ar@{->>}[d]\ar@{=}[r] &
 \bar\theta^*\mathcal{E}_r \ar@{->>}[d] \\
 {\overline{\zeta\eta'}\,}^*\mathcal{T}_{d,r}\ar@{=}[r] &
 \bar\eta'^*\bar\zeta^*\mathcal{T}_{d,r}\ar[r]^{\simeq} &
 \bar\eta'^*(\mathcal{T}_{d^\vee,r})_{Z}^\e\ar@{=}[r] &
 \bar\eta^*\mathcal{T}_{d^\vee,r}^\e \ar[r]^{\simeq} &
 \mathcal{T}_{d,r}^{\e\e}\ar@{=}[r] &
 \mathcal{T}_{d,r}
} %
\]
By the universal property of $Q_{d,r}$, $\zeta\eta'={\rm id}_{Q_{d,r}}$.

(ii) To show that $\eta'\zeta=\id_Z$, it suffices to show that $i\eta'\zeta=i$, or $\eta\zeta=i$.
We have the following commutative diagram of equivalent quotients:
\[
\CD{ %
{\overline{\eta\zeta}\,}^*\bar\vartheta^*\mathcal{E}_d^\vee \ar@{=}[r]\ar@{->>}[d] &
\bar\zeta^*\bar\eta^*\bar\vartheta^*\mathcal{E}_d^\vee \ar@{=}[r]\ar@{->>}[d] &
\bar\zeta^*\bar\phi^*\mathcal{E}_d^\vee \ar@{=}[r]\ar@{->>}[d] &
\bar\zeta^*\mathcal{E}_{d,r}^\vee \ar@{=}[r]\ar@{->>}[d] & (\bar\vartheta^*\mathcal{E}_d)_Z^\vee
\ar@{=}[r]\ar@{->>}[d] &
\bar i^*\bar\vartheta^*\mathcal{E}_d^\vee \ar@{->>}[d] \\
{\overline{\eta\zeta}\,}^*\mathcal{T}_{d^\vee,r} \ar@{=}[r] &
\bar\zeta^*\bar\eta^*\mathcal{T}_{d^\vee,r} \ar[r]^{\simeq} & \bar\zeta^*\mathcal{T}_{d,r}^\e \ar[r]^{\simeq}
& (\mathcal{T}_{d^\vee,r})_Z^{\e\e} \ar@{=}[r] & (\mathcal{T}_{d^\vee,r})_Z \ar@{=}[r] & \bar
i^*\mathcal{T}_{d^\vee,r}
} %
\]
By the universal property of $Q_{d^\vee,r}$, we have $\eta\zeta=i$ and hence $\eta'\zeta=\id_Z$.
\end{proof}

We now introduce a commutative diagram as follows.

\[
\xymatrix{ %
\ds\prod_{l=0}^{r-1}{}_{Q_{d,r}}\BBS_{\widehat{m}+l}(\theta^*\pi_*V_{\mathbb{P}^1_{Q_{r}}}^\vee\!(m),\theta^*\pi_*\mathcal{E}_{r}^\vee(m))
\ar@{^(.>}[d]^{\widetilde\eta}\ar[r] & Q_{d,r}
\ar@{^(->}[d]^{\eta}\ar[r]^{\phi} & Q_d \ar@{=}[d] \\
\ds\prod_{l=0}^{r-1}{}_{Q_{d^\vee,r}} \BBS_{\widehat{m}+l}(\pi_{*}V_{\mathbb{P}^1_{Q_{d^\vee,r}}}^\vee\!\!\!(m),\pi_{*}\mathcal{E}_{d^\vee\!\!,r}(m))\ar[r]\ar@{^(.>}[d]^{\widetilde\iota} &
Q_{d^\vee,r} \ar[r]^{\vartheta}\ar@{^(->}[d]^\iota &
Q_d \ar@{=}[d] \\
\ds\prod_{l=0}^{r-1}{}_{\Gr_{m,r}} \BBS_{\widehat{m}+l}(g^*\pi_*V_{\mathbb{P}^1_{Q_d}}^\vee(m),\mathcal{K}) \ar[r]\ar@{^(->}[d]^{\lambda} &
\Gr_{m,r}
\ar[r]^-{g}
& Q_d \ar@{=}[d] \\
\ds\prod_{l=0}^{r-1}{}_{Q_d} \BBS_{\widehat{m}+l}(\pi_*V_{\mathbb{P}^1_{Q_d}}^\vee(m),\pi_*\mathcal{E}_d^\vee(m))\ar[rr]& & Q_d
} %
\]
where
\begin{itemize}
\item $\Gr_{m,r}:=\Gr_{Q_d}({\widehat{m}+r},\pi_*\mathcal{E}_d^\vee(m))$, the relative Grassmannian over $Q_d$,
\item $g$ is the structure morphism,
\item $\mathcal{K}$ is the universal subbundle of $g^*\pi_*\mathcal{E}_d^\vee(m)$ on the Grassmannian $\Gr_{m,r}$,
\item $\iota$ is the closed embedding of Quot scheme into Grassmannian by Proposition (\ref{quotembedding}), and
\item $\eta$ is the closed embedding introduced by Proposition (\ref{eta}).
\end{itemize}

We now define the morphisms $\widetilde\eta$, $\widetilde\iota$ and $\lambda$, and show that they are
closed embeddings for all $m\gg 0$.

\textbf{Definition of $\widetilde\eta$.} For $m\gg 0$, we have the following identifications
\[
\begin{split}
&\ \quad\BBS_{\widehat{m}+l}\big(\theta^*\pi_*V_{\mathbb{P}^1_{Q_r}}^\vee(m),\theta^*\pi_*\mathcal{E}_{r}^\vee(m)\big) \stackrel{(1)}{=}\BBS_{\widehat{m}+l}(\pi_{*}V_{\mathbb{P}^1_{Q_{d,r}}}^\vee(m),\pi_{*}\bar\theta^*\mathcal{E}_{r}^\vee(m)) \\
&\stackrel{(2)}{=}\BBS_{\widehat{m}+l}(\pi_{*}\bar\eta^*V_{\mathbb{P}^1_{Q_{d^\vee,r}}}^\vee(m),\pi_{*}\bar\eta^*\mathcal{E}_{d^\vee,r}^\vee(m)) \stackrel{(3)}{=}\BBS_{\widehat{m}+l}(\eta^*\pi_{*}V_{\mathbb{P}^1_{Q_{d^\vee,r}}}^\vee(m),\eta^*\pi_{*}\mathcal{E}_{d^\vee,r}^\vee(m)) \\
&\stackrel{(4)}{=}Q_{d,r}\times_{Q_{d^\vee,r}}\BBS_{\widehat{m}+l}(\pi_{*}V_{\mathbb{P}^1_{Q_{d^\vee,r}}}^\vee(m),\pi_{*}\mathcal{E}_{d^\vee,r}^\vee(m))
\end{split}
\]
for $l=0,\dots,r-1$.
The equality (1) follows from the flatness of $\theta$, (2) follows from $\bar\theta^*\mathcal{E}_r=\bar\eta^*\mathcal{E}_{d^\vee,r}$, (3) follows from the base-change formula for $m\gg 0$,
\[
\pi_{*}\bar\theta^*\mathcal{E}_{r}^\vee(m)=
\pi_{*}\bar\eta^*\mathcal{E}_{d^\vee,r}(m)\cong\eta^*\pi_{*}\mathcal{E}_{d^\vee,r}(m),\quad
\pi_*V_{\mathbb{P}^1_{Q_{d,r}}}=\eta^*\pi_*V_{\mathbb{P}^1_{Q_{d^\vee,r}}}
\]
and the equalities (4) follow from the base-change property of $\PProj$.

So we have the following identification
\[
\BBS_{\widehat{m}+l}\big(\theta^*\pi_*V_{\mathbb{P}^1_{Q_r}}^\vee(m),\theta^*\pi_*\mathcal{E}_{r}^\vee(m)\big) = Q_{d,r}\times_{Q_{d^\vee,r}}\prod_{l=0}^{r-1}{}_{Q_{d^\vee,r}} \BBS_{\widehat{m}+l}(\pi_{*}V_{\mathbb{P}^1_{Q_{d^\vee,r}}}^\vee\!\!\!(m),\pi_{*}\mathcal{E}_{d^\vee,r}^\vee(m))
\]
and $\widetilde\eta$ is defined simply to be the projection
\[
Q_{d,r}\times_{Q_{d^\vee,r}}\prod_{l=0}^{r-1}{}_{Q_{d^\vee,r}}\BBS_{\widehat{m}+l}(\pi_{*}V_{\mathbb{P}^1_{Q_{d^\vee,r}}}^\vee\!\!\!(m), \pi_{*}\mathcal{E}_{d^\vee,r}^\vee(m)) \to \prod_{l=0}^{r-1}{}_{Q_{d^\vee,r}}\BBS_{\widehat{m}+l}(\pi_{*}V_{\mathbb{P}^1_{Q_{d^\vee,r}}}^\vee(m),\pi_{*}\mathcal{E}_{d^\vee,r}^\vee(m))
\]
That $\widetilde\eta$ is a closed embedding follows from that $\eta$ is a closed embedding.

\textbf{Definition of $\widetilde\iota$.} We have the following identifications
\[
\BBS_{\widehat{m}+l}(\pi_{*}V_{\mathbb{P}^1_{Q_{d^\vee,r}}}^\vee\!\!\!(m),\pi_{*}\mathcal{E}_{d^\vee,r}^\vee(m)) \stackrel{(5)}{=}
\BBS_{\widehat{m}+l}(\iota^*\pi_{*}V_{\mathbb{P}^1_{Q_{d}}}^\vee(m),\iota^*\mathcal{K})
 \stackrel{(6)}{=}Q_{d^\vee,r}\times_{\Gr_{m,r}}
\BBS_{\widehat{m}+l}(\pi_{*}V_{\mathbb{P}^1_{Q_{d}}}^\vee(m),\mathcal{K})
\]
where (5) follows from the definition of $\iota$ and (6) follows from the base-change property of
$\PProj$. So we have the identification
\[
\prod_{l=0}^{r-1}{}_{Q_{d^\vee,r}}\BBS_{\widehat{m}+l}(\pi_{*}V_{\mathbb{P}^1_{Q_{d^\vee,r}}}^\vee(m), \pi_{*}\mathcal{E}_{d^\vee,r}^\vee(m))
=Q_{d^\vee,r}\times_{\Gr_{m,r}}\prod_{l=0}^{r-1}{}_{\Gr_{m,r}} \BBS_{\widehat{m}+l}(\pi_{*}V_{\mathbb{P}^1_{Q_{d}}}^\vee(m),\mathcal{K})
\]
and $\widetilde\iota$ is just the projection
\[
 Q_{d^\vee,r}\times_{\Gr_{m,r}}\prod_{l=0}^{r-1}{}_{\Gr_{m,r}}
\BBS_{\widehat{m}+l}(\pi_{*}V_{\mathbb{P}^1_{Q_{d}}}^\vee(m),\mathcal{K}) \to\prod_{l=0}^{r-1}{}_{\Gr_{m,r}} \BBS_{\widehat{m}+l}(\pi_{*}V_{\mathbb{P}^1_{Q_{d}}}^\vee(m),\mathcal{K})
\]
Since $\iota$ is a closed embedding, so is $\widetilde\iota$.

\textbf{Definition of $\lambda$.} For each $l$, $\ds\bigwedge^{\widehat m+l+1}\mathcal{K}$ is a subbundle
of $\ds\bigwedge^{\widehat m+l+1}g^*\pi_*\mathcal{E}_d^\vee(m)$. This gives rise to a subbundle map
\[
\HHom(\bigwedge^{\widehat m+l+1}g^*\pi_*V_{\mathbb{P}^1_{Q_d}}^\vee(m),\bigwedge^{\widehat m+l+1}\mathcal{K})\to \HHom(\bigwedge^{\widehat m+l+1}g^*\pi_*V_{\mathbb{P}^1_{Q_d}}^\vee(m),\bigwedge^{\widehat m+l+1}g^*\pi_*\mathcal{E}_d^\vee(m))
\]
which induces a closed embeddings
\[
\BBS_{\widehat{m}+l}(g^*\pi_*V_{\mathbb{P}^1_{Q_d}}^\vee(m),\mathcal{K})\hookrightarrow
\BBS_{\widehat{m}+l}(g^*\pi_*V_{\mathbb{P}^1_{Q_d}}^\vee(m),g^*\pi_*\mathcal{E}_d^\vee(m))
\]
for $l=0,\dots,r-2$. Taking fibered products, we obtain an embedding
\begin{equation}\label{productembedding}
\begin{split}
\prod_{l=0}^{r-2}{}_{\Gr_{m,r}}\BBS_{\widehat{m}+l}(g^*\pi_*V_{\mathbb{P}^1_{Q_d}}^\vee(m),\mathcal{K}) &\hookrightarrow
\prod_{l=0}^{r-2}{}_{\Gr_{m,r}}\BBS_{\widehat{m}+l}(g^*\pi_*V_{\mathbb{P}^1_{Q_d}}^\vee(m),g^*\pi_*\mathcal{E}_d^\vee(m))= \\
&\prod_{l=0}^{r-2}{}_{Q_d}\BBS_{\widehat{m}+l}(\pi_*V_{\mathbb{P}^1_{Q_d}}^\vee(m),\pi_*\mathcal{E}_d^\vee(m)) \times_{Q_d}\Gr_{m,r}
\end{split}
\end{equation}
Note that we have the following commutative diagram
\[
\CDC{
  \BBS_{\widehat{m}+r-1}(g^*\pi_*V_{\mathbb{P}^1_{Q_d}}^\vee(m),\mathcal{K})
  \ar[r]\ar@{^(.>}[d]^-{\widetilde\delta} & \Gr_{m,r}
  \ar[r]^-{g}\ar@{^(->}[d]^-{\delta} & Q_d \ar@{=}[d] \\
  \BBS(h^*\ds\bigwedge^{\widehat m+r}\pi_*V_{\mathbb{P}^1_{Q_d}}^\vee(m),\mathcal{L})\ar[r]\ar@{^(->}[d]^-\sigma & \mathbb{P}\big(\ds\bigwedge^{\widehat{m}+r}\pi_*\mathcal{E}_d^\vee(m)\big)
  \ar[r]^-{h} & Q_d \ar@{=}[d] \\
  \BBS_{\widehat{m}+r-1}(\pi_*V_{\mathbb{P}^1_{Q_d}}^\vee(m),\pi_*\mathcal{E}_d^\vee(m))\ar[rr]& & Q_d
}
\]
where $h$ is the structure morphism, $\mathcal{L}$ is the universal subline bundle of
$h^*\ds\bigwedge^{\widehat m+r}\pi_*\mathcal{E}_d^\vee(m)$, and $\delta$ is the relative Pl\"ucker
embedding over $Q_d$. By definition of $\delta$, we have $\delta^*\mathcal{L}=\bigwedge^{\widehat m+r}\mathcal{K}$, hence
\[
\begin{split}
&\quad\ \BBS_{\widehat{m}+r-1}(g^*\pi_*V_{\mathbb{P}^1_{Q_d}}^\vee(m),\mathcal{K})=\BBS(\bigwedge^{\widehat m+r}g^*\pi_*V_{\mathbb{P}^1_{Q_d}}^\vee(m),\bigwedge^{\widehat m+r}\mathcal{K})\\
&=\BBS(\delta^*h^*\bigwedge^{\widehat m+r}\pi_*V_{\mathbb{P}^1_{Q_d}}^\vee(m),\delta^*\mathcal{L})=\BBS(h^*\bigwedge^{\widehat m+r}\pi_*V_{\mathbb{P}^1_{Q_d}}^\vee(m),\mathcal{L})\times_\mathbb{P} \Gr_{m,r}
\end{split}
\]
where $\mathbb{P}:=\mathbb{P}(\ds\bigwedge^{\widehat{m}+r}\pi_*\mathcal{E}_d^\vee(m))$. Thus
$\widetilde\delta$ is defined to be a projection and hence is also a closed embedding. The morphism
$\sigma$ is the relative Segr\'e embedding over $Q_d$:
\[
\begin{split}
&\quad\ \BBS\big(\ds\bigwedge^{\widehat m+r} h^*\pi_*V_{\mathbb{P}^1_{Q_d}}^\vee(m),\mathcal{L}\big) =\mathbb{P}\big(h^*\bigg(\ds\bigwedge^{\widehat m+r}\pi_*V_{\mathbb{P}^1_{Q_d}}^\vee(m)\bigg)^\vee\otimes \mathcal{L}\big) \\
& = \mathbb{P}(h^*\bigg(\ds\bigwedge^{\widehat m+r}\pi_*V_{\mathbb{P}^1_{Q_d}}^\vee(m)\bigg)^\vee)
=\mathbb{P}(\bigg(\ds\bigwedge^{\widehat m+r}\pi_*V_{\mathbb{P}^1_{Q_d}}^\vee(m)\bigg)^\vee)\times_{Q_d}
\mathbb{P}(\ds\bigwedge^{\widehat{m}+r}\pi_*\mathcal{E}_d^\vee(m)) \\
&\stackrel{\sigma}{\hookrightarrow} \mathbb{P}(\bigg(\ds\bigwedge^{\widehat m+r}\pi_*V_{\mathbb{P}^1_{Q_d}}^\vee(m)\bigg)^\vee\otimes
\bigwedge^{\widehat{m}+r}\pi_*\mathcal{E}_d^\vee(m))=\BBS_{\widehat m+r-1}(\pi_*V_{\mathbb{P}^1_{Q_d}}^\vee(m),\pi_*\mathcal{E}_d^\vee(m))
\end{split}
\]
Combining the embedding (\ref{productembedding}) with the composition of embeddings
$\sigma\widetilde\delta$, we obtain the embedding $\lambda$ as the composition of
\[
\begin{split}
&\quad\ \prod_{l=0}^{r-1}{}_{\Gr_{m,r}}\BBS_{\widehat{m}+l}(g^*\pi_*V_{\mathbb{P}^1_{Q_d}}^\vee(m),\mathcal{K}) \\
& \stackrel{(7)}{\hookrightarrow} (\prod_{l=0}^{r-2}{}_{Q_d}\BBS_{\widehat{m}+l}(\pi_*V_{\mathbb{P}^1_{Q_d}}^\vee(m), \pi_*\mathcal{E}_d^\vee(m))\times_{Q_d}\Gr_{m,r})\times_{\Gr_{m,r}}
\BBS_{\widehat{m}+r-1}(g^*\pi_*V_{\mathbb{P}^1_{Q_d}}^\vee(m),\mathcal{K}) \\
& =\prod_{l=0}^{r-2}{}_{Q_d}\BBS_{\widehat{m}+l}(\pi_*V_{\mathbb{P}^1_{Q_d}}^\vee(m),\pi_*\mathcal{E}_d^\vee(m))\times_{Q_d}
\BBS_{\widehat{m}+r-1}(g^*\pi_*V_{\mathbb{P}^1_{Q_d}}^\vee(m),\mathcal{K}) \\
&\stackrel{(8)}{\hookrightarrow}
\prod_{l=0}^{r-2}{}_{Q_d}\BBS_{\widehat{m}+l}(\pi_*V_{\mathbb{P}^1_{Q_d}}^\vee(m),\pi_*\mathcal{E}_d^\vee(m))\times_{Q_d} \BBS_{\widehat m+r-1}(\pi_*V_{\mathbb{P}^1_{Q_d}}^\vee(m),\pi_*\mathcal{E}_d^\vee(m)) \\
&= \prod_{l=0}^{r-1}{}_{Q_d}\BBS_{\widehat{m}+l}(\pi_*V_{\mathbb{P}^1_{Q_d}}^\vee(m),\pi_*\mathcal{E}_d^\vee(m))
\end{split}
\]
The embeddings (7) and (8) are induced by (\ref{productembedding}) and $\sigma\widetilde\delta$,
respectively.

Finally, we define $\beta=\lambda\widetilde\iota\widetilde\eta$. By construction, $\beta$ is an embedding. It remains to prove the commutativity of the diagram (\ref{ZembedDiagram}). For this, we use the universal homomorphism of the space of collineations. Consider the following diagram
\[
\xymatrix@R=0pc{
 & \mathring{Q}_{d,r}\ar@{_(->}[ld]_-\alpha\ar@{^(->}[rd]^j\ar[dd]^{\mathring{\varphi}} & \\
\ds\prod_{l=0}^{r-1}\BBS_{\widehat{m}+l}\big(\theta^*\pi_*V_{\mathbb{P}^1_{Q_r}}^\vee(m),\theta^*\pi_*\mathcal{E}_{r}^\vee(m)\big) \ar@{^(->}[dd]^\beta \ar@{-}[r] & \quad\ar[r] & Q_{d,r}\ar[dd]^\phi \\
 & \mathring{Z}_{d,r}\ar@{_(->}[ld]_-\gamma\ar@{^(->}[rd]^i & \\
\ds\prod_{l=0}^{r-1}\BBS_{\widehat{m}+l}\big(\pi_*V_{\mathbb{P}^1_{Q_d}}^\vee(m),\pi_*\mathcal{E}_{d}^\vee(m)\big) \ar[rr] & & Q_d
}
\]
For $l=0,\dots,r-1$, the universal homomorphism \[u:\bigwedge^{\widehat m+l+1}\theta^*\pi_*V_{\mathbb{P}^1_{Q_r}}^\vee(m)\to (\bigwedge^{\widehat m+l+1}\theta^*\pi_*\mathcal{E}_{r}^\vee(m))(1)\]
on $\BBS_{\widehat{m}+l}\big(\theta^*\pi_*V_{\mathbb{P}^1_{Q_r}}^\vee(m),\theta^*\pi_*\mathcal{E}_{r}^\vee(m)\big)$ is pullbacked by $\alpha$ to the homomorphism $j^*\bigwedge^{\widehat m+l+1}\rho_{r,m}$, i.e., we have a commutative diagram
\[
\CDR{
\alpha^*\ds\bigwedge^{\widehat m+l+1}\theta^*\pi_*V_{\mathbb{P}^1_{Q_r}}^\vee(m)\ar[r]^-{\alpha^*u}\ar@{=}[d] & \alpha^*((\ds\bigwedge^{\widehat m+l+1}\theta^*\pi_*\mathcal{E}_{r}^\vee(m))(1))\ar@{=}[d] \\
j^*\ds\bigwedge^{\widehat m+l+1}\theta^*\pi_*V_{\mathbb{P}^1_{Q_r}}^\vee(m)\ar[r] & j^*\ds\bigwedge^{\widehat m+l+1}\theta^*\pi_*\mathcal{E}_{r}^\vee(m)
}
\]
Similarly, we have
\[
\CDR{
\gamma^*\ds\bigwedge^{\widehat m+l+1}\pi_*V_{\mathbb{P}^1_{Q_d}}^\vee(m)\ar[r]^-{\gamma^*u}\ar@{=}[d] & \gamma^*((\ds\bigwedge^{\widehat m+l+1}\pi_*\mathcal{E}_{d}^\vee(m))(1))\ar@{=}[d] \\
i^*\ds\bigwedge^{\widehat m+l+1}\pi_*V_{\mathbb{P}^1_{Q_d}}^\vee(m)\ar[r] & i^*\ds\bigwedge^{\widehat m+l+1}\pi_*\mathcal{E}_{d}^\vee(m)
}
\]
By the construction of $\beta$, we have a commutative diagram
\[
\CD{
\beta^*\ds\bigwedge^{\widehat m+l+1}\pi_*V_{\mathbb{P}^1_{Q_d}}^\vee(m)\ar[rr]^-{\beta^*u}\ar@{=}[d] & & \beta^*((\ds\bigwedge^{\widehat m+l+1}\pi_*\mathcal{E}_{d}^\vee(m))(1))\ar@{=}[d] \\
\ds\bigwedge^{\widehat m+l+1}\theta^*\pi_*V_{\mathbb{P}^1_{Q_r}}^\vee(m)\ar[r]^-{u} & (\ds\bigwedge^{\widehat m+l+1}\theta^*\pi_*\mathcal{E}_{r}^\vee(m))(1)\ar[r] & (\ds\bigwedge^{\widehat m+l+1}\pi_*\mathcal{E}_{d,r}^\vee(m))(1)
}
\]
Combining the above, we get commutative diagrams for $l=0,\dots,r-1$:
\[
\CD{
\alpha^*\beta^*\ds\bigwedge^{\widehat m+l+1}\pi_*V_{\mathbb{P}^1_{Q_d}}^\vee(m)\ar[rr]^-{\alpha^*\beta^*u}\ar@{=}[d] & & \alpha^*\beta^*((\ds\bigwedge^{\widehat m+l+1}\pi_*\mathcal{E}_{d}^\vee(m))(1))\ar@{=}[d] \\
\alpha^*\ds\bigwedge^{\widehat m+l+1}\theta^*\pi_*V_{\mathbb{P}^1_{Q_r}}^\vee(m)\ar[r]^-{\alpha^*u}\ar@{=}[d] & \alpha^*((\ds\bigwedge^{\widehat m+l+1}\theta^*\pi_*\mathcal{E}_{r}^\vee(m))(1))\ar[r]\ar@{=}[d] & \alpha^*((\ds\bigwedge^{\widehat m+l+1}\pi_*\mathcal{E}_{d,r}^\vee(m))(1))\ar@{=}[d] \\
j^*\ds\bigwedge^{\widehat m+l+1}\theta^*\pi_*V_{\mathbb{P}^1_{Q_r}}^\vee(m)\ar[r]\ar@{=}[d] & j^*\ds\bigwedge^{\widehat m+l+1}\theta^*\pi_*\mathcal{E}_{r}^\vee(m)\ar[r] & j^*\ds\bigwedge^{\widehat m+l+1}\pi_*\mathcal{E}_{d,r}^\vee(m)\ar@{=}[d] \\
j^*\phi^*\ds\bigwedge^{\widehat m+l+1}\pi_*V_{\mathbb{P}^1_{Q_d}}^\vee(m)\ar[rr]\ar@{=}[d] & & j^*\phi^*\ds\bigwedge^{\widehat m+l+1}\pi_*\mathcal{E}_{d}^\vee(m)\ar@{=}[d] \\
\mathring{\varphi}^*i^*\ds\bigwedge^{\widehat m+l+1}\pi_*V_{\mathbb{P}^1_{Q_d}}^\vee(m)\ar[rr]\ar@{=}[d] & & \mathring{\varphi}^*i^*\ds\bigwedge^{\widehat m+l+1}\pi_*\mathcal{E}_{d}^\vee(m)\ar@{=}[d] \\
\mathring{\varphi}^*\gamma^*\ds\bigwedge^{\widehat m+l+1}\pi_*V_{\mathbb{P}^1_{Q_d}}^\vee(m)\ar[rr]^{\mathring{\varphi}^*\gamma^*u} & & \mathring{\varphi}^*\gamma^*((\ds\bigwedge^{\widehat m+l+1}\pi_*\mathcal{E}_{d}^\vee(m))(1))
}
\]
So we have $(\beta\alpha)^*u=\alpha^*\beta^*u=\mathring{\varphi}^*\gamma^*u=(\gamma\mathring{\varphi})^*u$
which implies $\beta\alpha=\gamma\mathring{\varphi}$ by the universal property of the space of collineations. This completes the proof of Proposition \ref{beta_exist}.

\end{document}